\documentclass[11pt]{amsart}
\usepackage{geometry}
\usepackage{graphicx}
\usepackage{amssymb}
\usepackage{epstopdf}
\usepackage{mathtools}
\usepackage{tikz-cd}
\usepackage{pb-diagram}
\usepackage{thmtools, thm-restate}

\newtheorem{theorem}{Theorem}[section]

\newtheorem{lemma}[theorem]{Lemma}
\newtheorem*{theorem*}{Theorem}

\usepackage{chngcntr}
\counterwithin{figure}{section}

\newtheorem{question}{Question}
\newtheorem{problem}[question]{Problem}
\newtheorem{conjecture}[question]{Conjecture}

\DeclareMathOperator{\Mod}{Mod}
\DeclareMathOperator{\Out}{Out}
\DeclareMathOperator{\Aut}{Aut}
\DeclareMathOperator{\Inn}{Inn}
\DeclareMathOperator{\Sp}{Sp}
\DeclareMathOperator{\GL}{GL}
\DeclareMathOperator{\SL}{SL}
\DeclareMathOperator{\SO}{SO}
\DeclareMathOperator{\A}{A}

\newtheorem{proposition}[theorem]{Proposition}

\title{Generating mapping class groups with elements of fixed finite order}

\author{Justin Lanier}
\address{School of Mathematics\\
Georgia Tech\\
686 Cherry St.\\
Atlanta, GA 30332}
\email{jlanier8@gatech.edu}



\begin{document}

\begin{abstract}
We show that for any $k \geq 6$ and $g$ sufficiently large, the mapping class group of a surface of genus $g$ can be generated by three elements of order $k$. We also show that this can be done with four elements of order $5$. We additionally prove similar results for some permutation groups, linear groups, and automorphism groups of free groups.
\end{abstract}

\maketitle
\section{Introduction}
\addtocontents{toc}{\protect\setcounter{tocdepth}{1}}

In this paper, we construct small generating sets for groups where all of the generators have the same finite order. Our main result is about $\Mod(S_g)$, the mapping class group of a closed, connected, and orientable surface of genus $g$.

\begin{theorem}
\label{thm:1}
Let $k \geq 6$ and $g \geq (k-1)^2+1$. Then $\Mod(S_g)$ is generated by three elements of order $k$. Also, $\Mod(S_g)$ is generated by four elements of order $5$ when $g \geq 8$.
\end{theorem}

Theorem \ref{thm:1} follows from a stronger but more technical result that we prove as Theorem \ref{thm:maintheorem}. Our generating sets for $\Mod(S_g)$ are constructed explicitly. In addition, the elements in any particular generating set are all conjugate to each other. Of course, attempting to construct generating sets consisting of elements of a fixed order $k$ only makes sense if $\Mod(S_g)$ contains elements of order $k$ in the first place. In the proof of Lemma \ref{thm:mainlemma}, we construct an element of any fixed order $k$ in $\Mod(S_g)$ whenever $g$ is sufficiently large.

We describe below prior work by several authors on generating $\Mod(S_g)$ with elements of small fixed finite order. Set alongside this prior work, a new phenomenon that emerges in our results is that the sizes of our generating sets for $\Mod(S_g)$ are not only independent of the genus of the surface, but they are also independent of the order of the elements.

After proving our results concerning mapping class groups, we prove analogous theorems for groups belonging to several other infinite families, each indexed by a parameter $n$:

\begin{itemize}
    \item $\A_n$, the alternating group on $n$ symbols,
    \item $\Aut^+(F_n)$, the special automorphism group of the free group of rank $n$,
    \item $\Out^+(F_n)$, the special outer automorphism group of the free group of rank $n$,
    \item $\SL(n,\mathbb{Z})$, the integral special linear group of $n \times n$ matrices,
    \item $\Sp(2n,\mathbb{Z})$, the integral symplectic group of $2n \times 2n$ matrices.
\end{itemize}

Note that $\A_n$, $\Aut^+(F_n)$, $\Out^+(F_n)$, and $\SL(n,\mathbb{Z})$ are index 2 subgroups of the symmetric group $\Sigma_n$, the automorphism group of a free group $\Aut(F_n)$, the outer automorphism group of a free group $\Out(F_n)$, and the integral general linear group $\GL(n,\mathbb{Z})$, respectively. Similarly, $\Mod(S_g)$ is an index 2 subgroup of the extended mapping class group $\Mod^\pm(S_g)$, which includes orientation-reversing mapping classes. Since we are interested in generating groups with elements of some fixed finite order, it is more natural to avoid the parity issues that arise in generating these supergroups. Adding a single orientation-reversing element to any of the generating sets we produce will promote it to a generating set for its respective supergroup.

Our theorems for these other families of groups are as follows:

\begin{theorem}
\label{thm:perm}
Let $k \geq 3$ and $n \geq k$. Then three elements of order $k$ suffice to generate $\Sigma_n$ when $k$ is even and to generate $\A_n$ when $k$ is odd. Further, let $k \geq 3$ and $n \geq k+2$. Then four elements of order $k$ suffice to generate $\A_n$ when $k$ is even.
\end{theorem}

\begin{theorem}
\label{thm:others}
Let $k \geq 5$ and $n \geq 2(k-1)$. Then $\Aut^+(F_n)$, $\Out^+(F_n)$, and $\SL(n,\mathbb{Z})$ are each generated by eight elements of order $k$. When $k \geq 6$ and $n \geq 2k$, seven elements of order $k$ suffice.
\end{theorem}

\begin{theorem}
\label{thm:Sp}
Let $k \geq 5$ and $n \geq 2(k-1)(k-3)$. Then $\Sp(2n,\mathbb{Z})$ is generated by four elements of order $k$. When $k \geq 6$ and $n \geq 2((k-1)^2+1)$, then three elements of order $k$ suffice.
\end{theorem}

We mention that the proof of Theorem \ref{thm:others} additionally shows that there is a straightforward extension of our mapping class group results to $\Mod(S_{g,1})$, the mapping class groups of once-punctured surfaces.

\subsection*{Outline of the paper}
In Section 2 we show how to construct elements of order $k$ in $\Mod(S_g)$ for sufficiently large values of $g$.
From here, our strategy for constructing generating sets for $\Mod(S_g)$ unfolds as follows. Humphries \cite{Humphries} showed that the Dehn twists about the $2g+1$ curves in Figure \ref{fig:Hum} generate $\Mod(S_g)$. In order to show that a collection of elements generates $\Mod(S_g)$, it suffices to show two things: that a Dehn twist about some Humphries curve can be written as a product in these elements, and that all $2g+1$ Humphries curves are in the same orbit under the subgroup generated by these elements. If these two things hold, it follows that the Dehn twist about each of the Humphries curves may be written as a product in the elements, and therefore the elements generate $\Mod(S_g)$. We show how to write a Dehn twist as a product in elements of order $k$ in Section 3, and we construct elements of order $k$ that generate a subgroup under which all Humphries curves are in the same orbit in our proof of Theorem \ref{thm:maintheorem} in Sections 4 and 5.

Our results about permutation groups are independent of the mapping class group result, and we prove these in Section 6. In Section 7 we combine our results for $\Mod(S_g)$ and $\A_n$ to obtain our result for $\Aut^+(F_n)$. This result in turn yields the theorems for $\Out^+(F_n)$ and $\SL(n, \mathbb{Z})$. Finally, our result for $\Sp(2n,\mathbb{Z})$ follows directly from our mapping class group result. We close in Section 8 with some further questions.

\begin{figure}[ht]
\centering
\includegraphics[width=\textwidth]{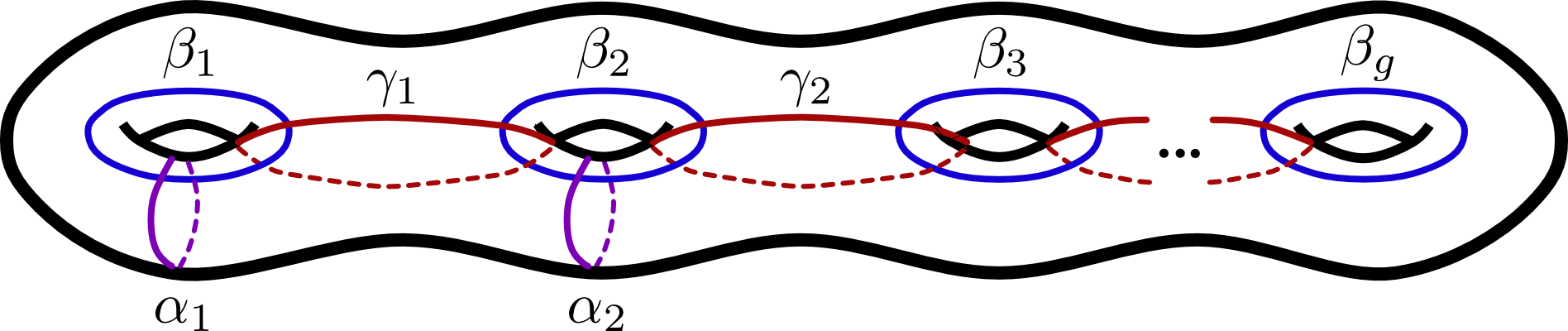}
\caption{The $2g+1$ Humphries curves in $S_g$.}
\label{fig:Hum}
\end{figure}
\subsection*{Background and prior results for \boldmath{$\Mod(S_g)$}} 
Let $S_g$ be a closed, connected, and orientable surface of genus $g$. The mapping class group $\Mod(S_g)$ is the group of homotopy classes of homeomorphisms of $S_g$. The most commonly-used generating sets for $\Mod(S_g)$ consist of Dehn twists, which have infinite order. Dehn \cite{Dehn} showed that $2g(g-1)$ Dehn twists generate $\Mod(S_g)$, and Lickorish \cite{Lickorish} showed that $3g+1$ Dehn twists suffice. Humphries \cite{Humphries} showed that only $2g+1$ Dehn twists are needed, and he also showed that no smaller set of Dehn twists can generate $\Mod(S_g)$.

There have also been many investigations into constructing generating sets for $\Mod(S_g)$ that include or even consist entirely of periodic elements. For instance, Maclachlan \cite{Maclachlan} showed that $\Mod(S_g)$ is normally generated by a set of two periodic elements that have orders $2g+2$ and $4g+2$, and McCarthy and Papadopoulos \cite{McPap} showed that $\Mod(S_g)$ is normally generated by a single involution (element of order 2) for $g \geq 3$. Korkmaz \cite{Korkmaz} showed that $\Mod(S_g)$ is generated by two elements of order $4g+2$ for $g \geq 3$.

Luo \cite{Luo} constructed the first finite generating sets for $\Mod(S_g)$ where the generators all have the same finite order, and where this order is independent of $g$. Luo's generating sets consisted of $6(2g+1)$ involutions, given that $g \geq 3$. Luo asked whether there exists a universal upper bound for the number of involutions required to generate $\Mod(S_g)$. Brendle and Farb \cite{Brendle2004187} showed that six involutions suffice to generate $\Mod(S_g)$, again for $g \geq 3$. Kassabov \cite{Kassabov} sharpened this result by showing that only five involutions are needed for $g \geq 5$ and only four are needed for $g \geq 7$. Monden \cite{PSP:8772025} showed that $\Mod(S_g)$ can be generated by three elements of order 3 and by four elements of order 4, each for $g \geq 3$. Recently Yoshihara \cite{Yoshihara} has shown that $\Mod(S_g)$ can be generated by three elements of order 6 when $g \geq 10$ and by four elements of order 6 when $g \geq 5$.

Much work has been done to establish when elements of a particular finite order exist in $\Mod(S_g)$. In his paper, Monden noted that for all $g \geq 1$, $\Mod(S_g)$ contains elements of orders 2, 3, and 4. Aside from order 6, elements of larger orders do not always exist in $\Mod(S_g)$. For example, $\Mod(S_3)$ contains no element of order 5 and $\Mod(S_4)$ contains no element of order 7.

Determining the orders of the periodic elements in $\Mod(S_g)$ for any particular $g$ is a solved problem, at least implicitly. In fact, this is even true for the determining the conjugacy classes of periodic elements in $\Mod(S_g)$. Ashikaga and Ishizaka \cite{ashikaga2002} listed necessary and sufficient criteria for determining the conjugacy classes in $\Mod(S_g)$ for any particular $g$. The criteria are number theoretic and consist of the Riemann--Hurwitz formula, the upper bound of $4g+2$ due to Wiman, an integer-sum condition on the valencies of the ramification points due to Nielsen, and several conditions on the least common multiple of the ramification indices that are due to Harvey. Ashikaga and Ishizaka also gave lists of the conjugacy classes of periodic elements in $\Mod(S_1)$, $\Mod(S_2)$, and $\Mod(S_3)$. Hirose \cite{hirose2010} gave a list of the conjugacy classes of periodic elements in $\Mod(S_4)$. Broughton \cite{Broughton} listed criteria for determining actions of finite groups on $S_g$, and hence for determining conjugacy classes of finite subgroups of $\Mod(S_g)$. Broughton also gave a complete classification of actions of finite groups on $S_2$ and $S_3$. Kirmura \cite{Kimura} gave a complete classification for $S_4$.

Several results have been proved about guaranteeing the existence of elements of order $k$ in $\Mod(S_g)$ for sufficiently large $g$. Harvey \cite{Harveyprime} showed that $\Mod(S_g)$ contains an element of order $k$ whenever $g \geq (k^2-1)/2$. Glover and Mislin \cite{GM} showed that $\Mod(S_g)$ contains an element of order $k$ whenever $g > (2k)^2$. Tucker \cite{Tucker} gave necessary and sufficient conditions for the existence of an element of order $k$ in $\Mod(S_g)$ that is realizable by a rotation of $S_g$ embedded in $\mathbb{R}^3$. Using this characterization, Tucker showed that for any $k$ and for sufficiently large $g$, $\Mod(S_g)$ contains an element of order $k$ that is realizable by a rotation of $S_g$ embedded in $\mathbb{R}^3$. We give a proof of this fact in Lemma \ref{thm:mainlemma}.

\subsection*{Background and prior results for other groups.}

The symmetric group $\Sigma_n$ is the group of permutations on $n$ symbols. The alternating group $\A_n$ is the unique index 2 subgroup of $\Sigma_n$ and consists of the even permutations. There are of course many results about generating sets for $\Sigma_n$ and $\A_n$. We provide here a few examples about generating sets consisting of a universally bounded number of elements of small or fixed order. Miller \cite{Miller1} showed that except for a few cases where $n\leq8$, every $\Sigma_n$ and $\A_n$ is generated by two elements, one of order 2 and one of order 3. Later, Miller \cite{Miller2} showed that whenever $\A_n$ contains an element of order $k>3$, the group may be generated by two elements, one of order 2 and one of order $k$. He also showed that the same holds for $\Sigma_n$, except for the case where $k=4$ and $n=6$ and in the cases where $k>3$ is an odd prime and $n = 2k-1$. Nuzhin \cite{Nuzhin} showed that $\A_n$ is generated by three involutions (where two commute) if and only if $n \geq 9$ or $n=5$. This implies that there is a universal upper bound on the number of involutions needed to generate $\A_n$ whenever they generate $\A_n$ at all. Annin and Maglione \cite{Econ} determined that $\max\{2,\left\lceil(n-1)/(k-1)\right\rceil\}$ is the smallest number of $k$-cycles needed to generate $\Sigma_n$ when $k$ is even and to generate $\A_n$ when $k$ is odd.

Turning to automorphism groups of free groups, let $F_n$ be the free group on $n$ generators $\{a_1,\dots,a_n\}$. Then $\Aut(F_n)$ is the automorphism group of $F_n$, $\Inn(F_n)$ is the inner automorphism group of $F_n$, and $\Out(F_n)=\Aut(F_n) / \Inn(F_n)$ is the outer automorphism group of $F_n$. The abelianization of each of $\Aut(F_n)$ and $\Out(F_n)$ is $\GL(n,\mathbb{Z})$. The preimages of the subgroup $\SL(n,\mathbb{Z})$ in $\Aut(F_n)$ and $\Out(F_n)$ under the abelianization map are respectively $\Aut^+(F_n)$ and $\Out^+(F_n)$. These are each additionally characterized as being the unique index 2 subgroups of $\Aut(F_n)$ and $\Out(F_n)$. Note that since $\Aut(F_n)$ surjects onto $\Out(F_n)$, a generating set for $\Aut(F_n)$ also yields a generating set for $\Out(F_n)$. The same relationship holds between $\Aut^+(F_n)$ and $\Out^+(F_n)$.

Nielsen \cite{Nielsen} gave the first finite presentations for $\Aut(F_n)$ and $\Out(F_n)$. His generators consisted of permutations of the generators of $F_n$ as well as an inversion (e.g. $a_1 \rightarrow a_1^{-1}$) and a transvection (e.g. $a_1 \rightarrow a_1a_2$) of generators of $F_n$. Notice that an inversion has order 2 and a transvection has infinite order. Magnus \cite{Magnus} showed that $\Out(F_n)$ is generated by an inversion of a generator of $F_n$ along with the collection of left and right transvections of generators of $F_n$. Gersten \cite{Gersten} gave a presentation for $\Aut^+({F_n})$ and in particular showed that the collection of left and right transvections generate $\Aut^+({F_n})$.

There are several results about generating $\Aut(F_n)$ with elements of finite order. Neumann \cite{Neumann} gave a presentation for $\Aut(F_n)$ with generators of order at most $n$. Armstrong, Forrest, and Vogtmann \cite{Vogtmann} gave a presentation for $\Aut(F_n)$ where the generating set is a finite set of involutions and where the number of involutions required grows with $n$. In a result similar to that of Brendle and Farb for $\Mod(S_g)$, Zucca \cite{Zucca} showed that $\Aut(F_n)$ is generated by three involutions for $n \geq 5$. Tamburini and Wilson \cite{TW} showed that $\Aut(F_n)$ is generated by two elements, one of order 2 and one of order 3, for $n \geq 18$.

The group $\SL(n, \mathbb{Z})$ is the group of $n \times n$ matrices with integer entries and determinant 1. Gustafson \cite{Gustafson} showed that $\SL(n, \mathbb{Z})$ is generated by involutions. The group $\Sp(2n,\mathbb{Z})$ is the group of $2n \times 2n$ matrices with integer entries that respect the standard symplectic form on $\mathbb{R}^{2n}$. Ishibashi \cite{Ishibashi} showed that $\Sp(2,\mathbb{Z})$ can be generated by two elements, one of order 4 and the other of order 3, and that $\Sp(2n,\mathbb{Z})$, $n>1$, can be generated by two elements, one of order 2 and the other of order $12(n-1)$ if $n$ is even and of order $6(n-1)$ if $n$ is odd.

\subsection*{Acknowledgments}
The author would like to thank Dan Margalit for his guidance, support, feedback, and encouragement throughout this project. The author would also like to thank Shane Scott for feedback and helpful conversations. Thanks also go to Martin Kassabov, both for his comments and for encouraging the author to sharpen the mapping class group result to three elements. The author also thanks Allen Broughton, John Dixon, John Etnyre, Benson Farb, Bill Harvey, Naoyuki Monden, Bal\'azs Strenner, Tom Tucker, and Josephine Yu for comments on a draft of this paper. The author was
partially supported by the NSF grant DGE-1650044.

\section{Constructing elements of a given order in $\Mod(S_g)$}
In this section we construct elements of order $k$ in $\Mod(S_g)$ whenever $g$ is sufficiently large. We will use elements that are conjugate to these elements when we build the generating sets of Theorem \ref{thm:maintheorem}.

\begin{lemma}
\label{thm:mainlemma}
Let $k \geq 2$. Then $\Mod(S_g)$ contains an element of order $k$ whenever $g>0$ can be written as $ak+b(k-1)$ with $a,b \in \mathbb{Z}_{\geq 0}$ or as $ak+1$ with $a \in \mathbb{Z}_{>0}$.
\end{lemma}

This result was proved by Tucker \cite{Tucker}, who additionally showed that these number theoretic conditions are necessary for an element of $\Mod(S_g)$ to be realizable by a rotation of $S_g$ embedded in $\mathbb{R}^3$.
\begin{proof}
In Figure \ref{fig:54final} we depict two ways of embedding a surface in $\mathbb{R}^3$ so that it has $k$-fold rotational symmetry. First, we can embed a surface of genus $k$ in $\mathbb{R}^3$ so that it has a rotational symmetry of order $k$ by evenly spacing $k$ handles about a central sphere. We can also embed a surface of genus $k-1$ into $\mathbb{R}^3$ so that it has a rotational symmetry of order $k$, as follows. Arrange two spheres along an axis of rotation and remove $k$ disks from each sphere, evenly spaced along the equator of each. Then connect pairs of boundary components, one from each sphere, with a cylinder. This can be done symmetrically so that a rotation by $2\pi/k$ permutes the cylinders cyclically.

\begin{figure}[ht]
\centering
\includegraphics[width=280pt]{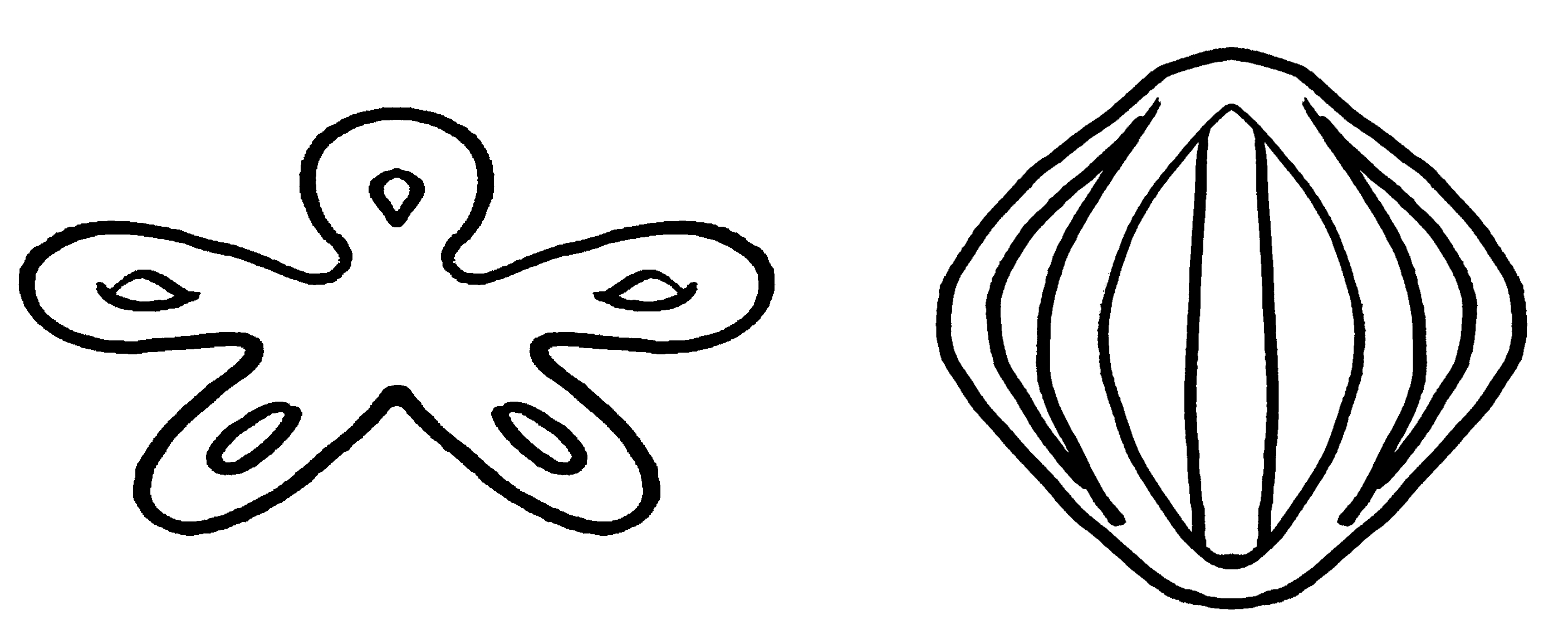}
\caption{Embeddings of $S_5$ and $S_4$ with rotational symmetry of order 5.}
\label{fig:54final}
\end{figure}

\begin{figure}[ht]
\centering
\includegraphics[width=280pt]{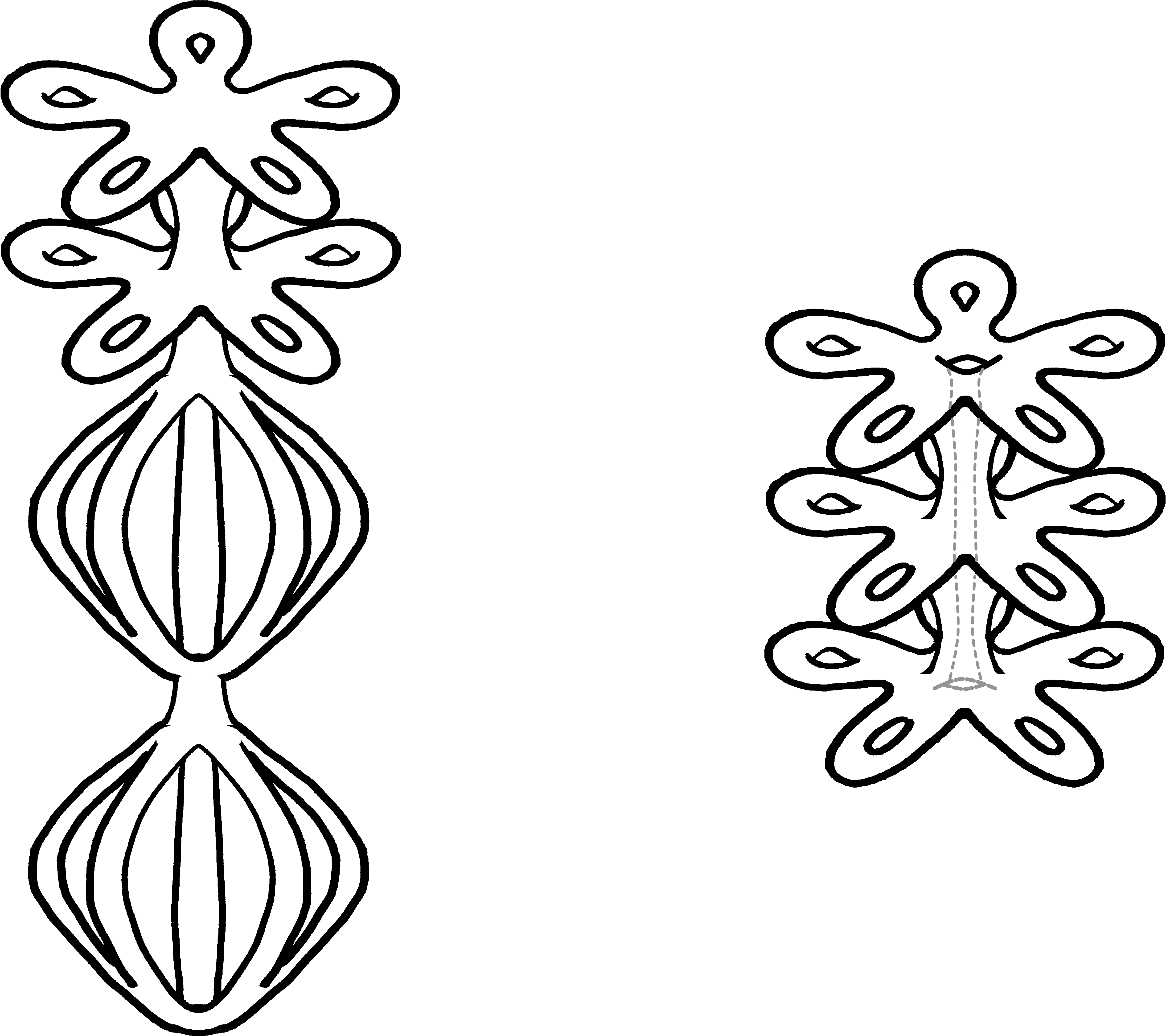}
\caption{Embeddings of $S_{18}$ and $S_{16}$ with rotational symmetry of order 5.}
\label{fig:stackandplus1}
\end{figure}

We can use these two types of embeddings to construct embeddings of surfaces of higher genus that also have rotational symmetry of order $k$. Whenever $g=ak+b(k-1)$, we can construct an embedding of $S_g$ in $\mathbb{R}^3$ by taking a connected sum of surfaces of genus $k$ and $k-1$ along their axis of rotational symmetry. See the left of Figure \ref{fig:stackandplus1} for an example. Rotating a surface embedded in this way by $2\pi/k$ produces an element of $\Mod(S_g)$ of order $k$ for any genus $g=ak+b(k-1)$. That an element so formed does not have order less than $k$ can be seen by the element's action on homology.

In order to produce elements of order $k$ in the case where $g=ak+1$, we first construct a surface of genus $ak$ with $k$-fold rotational symmetry by the above construction. We can modify this surface to increase its genus by 1 while preserving its symmetry as follows. See the right of Figure \ref{fig:stackandplus1}. The axis of a genus $ak$ surface intersects the surface at two points---at the top and the bottom. Removing an invariant disk around each of these points creates two boundary components. Connecting the two boundary components with a cylinder yields an embedding of a surface of genus $ak+1$ with $k$-fold symmetry.
\end{proof}

By way of some elementary number theory, we show that all sufficiently large integers have either the form $ak+b(k-1)$ or the form $ak+1$.

\begin{lemma}
If $k \geq 5$ and $g \geq (k-1)(k-3)$, then $g$ can either be written in the form $ak+b(k-1)$ with $a,b \in \mathbb{Z}_{\geq 0}$ or in the form $ak+1$ with $a \in \mathbb{Z}_{>0}$.
\label{lemma:frobenius}
\end{lemma}

\begin{proof}[Proof]
All integers at least $(k-1)(k-2)$ can be written in the form $ak+b(k-1)$ with $a,b \in \mathbb{Z}_{\geq 0}$ by the solution to the Frobenius coin problem. Further, every number from $(k-1)(k-3)$ to $k(k-3)$ can also be written as a sum of $k$'s and $k-1$'s. Start with $k-3$ copies of $k-1$ and replace the $k-1$'s one at a time by $k$'s. Finally, $k(k-3)+1=(k-1)(k-2)-1$ is of the form $ak+1$.
\end{proof}

In addition to producing elements of order $k$ in the stable range $g \geq (k-1)(k-3)$, we note that the construction given in Lemma \ref{thm:mainlemma} is also valid for approximately half of the values of $g$ less than $(k-3)(k-1)$. Specifically, $(k^2-3k-4)/2$ of these $k^2-4k+2$ smaller values of $g$ are either of the form $ak+b(k-1)$ or $ak+1$. This amount is simply $\sum_{i=3}^{k-2}i$, since $\{k-1,k,k+1\}$ is the first run of numbers of the given forms and $\{(k-4)(k-1),...,(k-4)k+1)\}$ is the last run less than $(k-1)(k-3)$.

Also, note that Lemma \ref{thm:mainlemma} includes the cases where $k$ is 2, 3, or 4. However, the construction we use to create the generating sets of Theorem \ref{thm:maintheorem} does not work for these small orders. However, these values of $k$ are those already treated by Luo, Brendle and Farb, Kassabov, and Monden in their work on generating sets for $\Mod(S_g)$.

\section{Building a Dehn twist}

In this section, we show that a Dehn twist in $\Mod(S_g)$ about a nonseparating curve may be written as a product in four elements whenever these act on small collections of curves in a specified way. In fact, even fewer than four elements will suffice as long as products in these elements act on the collections of curves as specified. In our proof, we follow the argument that Luo \cite{Luo} gave for writing a Dehn twist as a product of involutions, as well as the pair swap argument made by Brendle and Farb \cite{Brendle2004187}.

We write $T_{c}$ for the (left) Dehn twist about the curve $c$. Recall the \textit{lantern relation} that holds among Dehn twists about seven curves arranged in a sphere with four boundary components. In the left of Figure \ref{fig:lantern} we depict a sphere with four boundary components $L$ that is a subsurface of $S_g$. Singling out this particular lantern is convenient for our proof of Theorem \ref{thm:maintheorem}. Note that $S_g \setminus L$ is connected. Seven curves lie in $L$ in a lantern arrangement, and several of these are Humphries curves. We will call these seven curves \textit{lantern curves}. We have the following lantern relation:
\begin{align*}
T_{\alpha_1}T_{\alpha_2}T_{x_1}T_{\gamma_2} &= T_{\gamma_1}T_{x_3}T_{x_2}.
\end{align*}

Finally, recall that for a Dehn twist $T_c$ and a mapping class $f$, we have $fT_cf^{-1}=T_{f(c)}$.

\begin{lemma}Suppose we are given the subsurface $L$ in $S_g$ and elements $f$, $g$, and $h$ in $\Mod(S_g)$ such that
\begin{align*}
{f}(\gamma_1) &= \gamma_2\\
{g}(x_3,x_1) &= (\gamma_1,\gamma_2)\\
{h}(x_2,\alpha_2) &= (\gamma_1,\gamma_2).
\end{align*}
Then the Dehn twist $T_{\alpha_1}$ may be written as a product in $f$, $g$, $h$, an element conjugate to $f$, and their inverses.
\label{lemma:four}
\end{lemma}

While this lemma is stated for a specific Dehn twist by way of a specific lantern, the result holds for other Dehn twists by the \textit{change of coordinates principle}: if two collections of curves on a surface $S_g$ are given by the same topological data, then there exists a homeomorphism of $S_g$ to itself that maps the first collection of curves to the second. Details are given by Farb and Margalit \cite{Primer}.

\begin{proof}Since Dehn twists about nonintersecting curves commute, one form of the lantern relation for $L$ is
\begin{align*}
T_{\alpha_1} &= (T_{\gamma_1}T^{-1}_{\gamma_2})(T_{x_3}T^{-1}_{x_1})(T_{x_2}T^{-1}_{\alpha_2}).
\end{align*}

Applying the assumptions on $g$ and $h$ yields
\begin{align*}
T_{\alpha_1} &= (T_{\gamma_1}T^{-1}_{\gamma_2})(g^{-1}(T_{\gamma_1}T^{-1}_{\gamma_2})g)(h^{-1}(T_{\gamma_1}T^{-1}_{\gamma_2})h).
\end{align*}

Applying the assumption on $f$ and regrouping yields
\begin{align*}
T_{\alpha_1} &= ((f^{-1}T_{\gamma_2}f)T^{-1}_{\gamma_2})(g^{-1}(f^{-1}T_{\gamma_2}f)T^{-1}_{\gamma_2}g)(h^{-1}(f^{-1}T_{\gamma_2}f)T^{-1}_{\gamma_2}h)\\
&= (f^{-1}(T_{\gamma_2}fT^{-1}_{\gamma_2}))(g^{-1}f^{-1}(T_{\gamma_2}fT^{-1}_{\gamma_2})g)(h^{-1}f^{-1}(T_{\gamma_2}fT^{-1}_{\gamma_2})h).
\end{align*}

We have written $T_{\alpha_1}$ as a product in $f$, $g$, $h$, $T_{\gamma_2}fT^{-1}_{\gamma_2}$, and their inverses.
\end{proof}

\begin{figure}[ht]
\centering
\includegraphics[width=\textwidth]{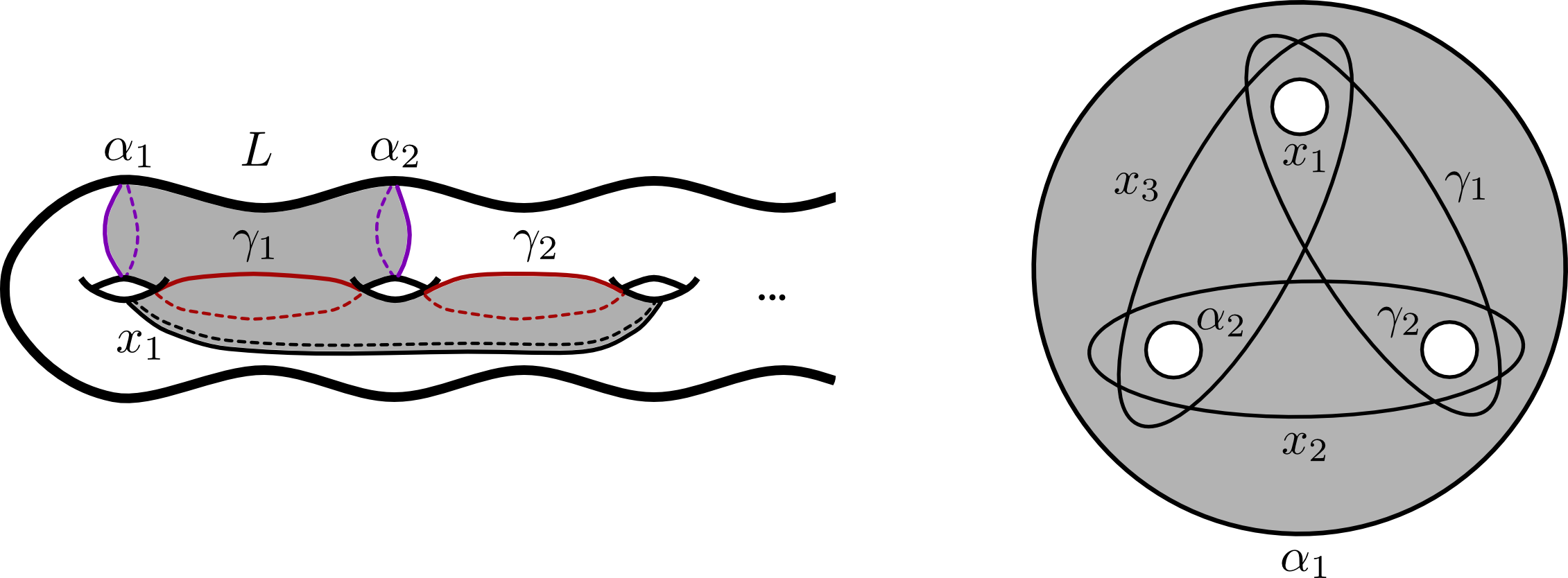}
\caption{The subsurface $L$ with five lantern curves drawn in. Two additional lantern curves $x_2$ and $x_3$ are omitted for clarity but are determined uniquely by the first five.}
\label{fig:lantern}
\end{figure}

Note that if $f$ has order $k$, then so does $T_{\gamma_2}fT^{-1}_{\gamma_2}$ since it is a conjugate of $f$. Finally, notice that we required very little of $f$, $g$, and $h$ in this argument---only that they map one specific curve or one specific pair of curves to another. We will take advantage of this flexibility in the proof of Theorem \ref{thm:maintheorem}.

\section{Generating $\Mod(S_g)$ with four elements of order $k$}

In this section and in the following section we prove the two parts of the following theorem, which is our main technical result.
\begin{theorem}
\label{thm:maintheorem}(1) Let $k \geq 5$ and let $g>0$ be of the form $ak+b(k-1)$ with $a,b \in \mathbb{Z}_{\geq 0}$ or of the form $ak+1$ with $a \in \mathbb{Z}_{>0}$. Then $\Mod(S_g)$ is generated by four elements of order $k$.\\

(2) Let $k \geq 8$ or $k=6$ and let $g>0$ be of the form $ak+b(k-1)$ with $a,b \in \mathbb{Z}_{\geq 0}$. Then $\Mod(S_g)$ is generated by three elements of order $k$. If instead $k=7$ and $g$ is of the form $7+7a+6b$ with $a,b \in \mathbb{Z}_{\geq 0}$, then three elements of order 7 also suffice.
\end{theorem}

Theorem \ref{thm:1} in the introduction follows directly from Theorem \ref{thm:maintheorem}, along with Lemma \ref{lemma:frobenius} and the observation that any $g \geq (k-1)^2+1$ may be written as a sum of $k$'s and $(k-1)$'s with at least one summand equal to $k$.

In this section we prove the statement about generating with four elements. In order to illustrate our construction, we depict the particular case $k=5$ and $g=18$ in Figure \ref{fig:ghat}. In what follows, a \textit{chain} of curves on a surface is a sequence of curves $c_1, \dots, c_t$ such that pairs of consecutive curves in the sequence intersect exactly once and each other pair of curves is disjoint.

\begin{proof} We begin with the case where $g=ak+b(k-1)$ and treat the case where $ak+1$ with a small modification at the end of the proof. Since $g=ak+b(k-1)$, we have a $k$-fold symmetric embedding of $S_g$ in $\mathbb{R}^3$ as constructed in Lemma \ref{thm:mainlemma}. Call this embedded surface $\Sigma_g$ and let it be comprised of $a$ surfaces of genus $k$ followed by $b$ surfaces of genus $k-1$. Let $\sigma_1$ through $\sigma_{a+b}$ denote these $k$-symmetric subsurfaces of $\Sigma_g$. Let $r$ be a rotation of $\Sigma_g$ by $2\pi/k$ about its axis.

We will construct our desired elements by mapping $S_g$ to $\Sigma_g$, performing a rotation $r$, and then mapping back to $S_g$. In doing so we will specify how individual curves map over and back again, and so control how curves are permuted among themselves. In order to construct these maps, it is convenient to label curves on $S_g$ and $\Sigma_g$ as follows. Take on the one hand the usual embedding of the Humphries curves in $S_g$ as shown in Figure \ref{fig:Hum} and the upper-left of Figure \ref{fig:ghat}. We will refer to the $\alpha_i$, $\beta_i$ and $\gamma_i$ curves as $\alpha$ \textit{curves}, $\beta$ \textit{curves}, and $\gamma$ \textit{curves}, respectively. The Humphries curves consist of a chain of $2g-1$ curves that alternate between $\beta$ curves and $\gamma$ curves as well as two additional $\alpha$ curves. On the other hand, take the $k$-fold symmetric embedded surface $\Sigma_g$ and embed in each $\sigma_i$ a chain of curves of length $2g_i-1$, where $g_i$ is the genus of $\sigma_i$. See Figure \ref{fig:54top} and the upper-left of Figure \ref{fig:ghat}. We label the curves in these chains also as $\beta$ and $\gamma$ curves and note that they are embedded so that $r(\beta_i)=\beta_{i+1}$, $1 \leq i \leq g_{i}-1$, and $r(\gamma_i)=\gamma_{i+1}$, $1 \leq i \leq g_{i}-2$. We will use these labels as ``local coordinates"---saying, for instance, ``the $\beta_2$ curve in $\sigma_3$." In $\sigma_1$ we additionally embed two $\alpha$ curves, $\alpha_1$ and $\alpha_2$, such that each respectively intersects $\beta_1$ and $\beta_2$ once, intersects no other curves, and $r(\alpha_1)=\alpha_2$.

We are now prepared to define three maps $\hat{f}$, $\hat{g}$, and $\hat{h}$ from $S_g$ to $\Sigma_g$. We will use these homeomorphisms to define three mapping classes of the form $\hat{f}^{-1}r\hat{f}$ and will show that these mapping classes (1) have order $k$, (2) satisfy Lemma \ref{lemma:four} and (3) put the Humphries curves into the same orbit.

We first construct a homeomorphism $\hat{f}$. The $\beta$ and $\gamma$ Humphries curves in $S_g$ form a chain of length $2g-1$. By removing some of the $\gamma$ curves from this chain, we form $a+b$ smaller chains. The first $a$ chains will be $2k-1$ curves long and the last $b$ chains will be $2k-3$ curves long. We accomplish this by removing every $k$th $\gamma$ curve up to $\gamma_{ak}$, and then every $(k-1)$st $\gamma$ curve thereafter. We call these the \textit{excluded} $\gamma$ curves. Call the resulting chains $F_i$, keeping their sequential order. We add to $F_1$ the curves $\alpha_1$ and $\alpha_2$.

\begin{figure}[ht]
\centering
\includegraphics[width=400pt]{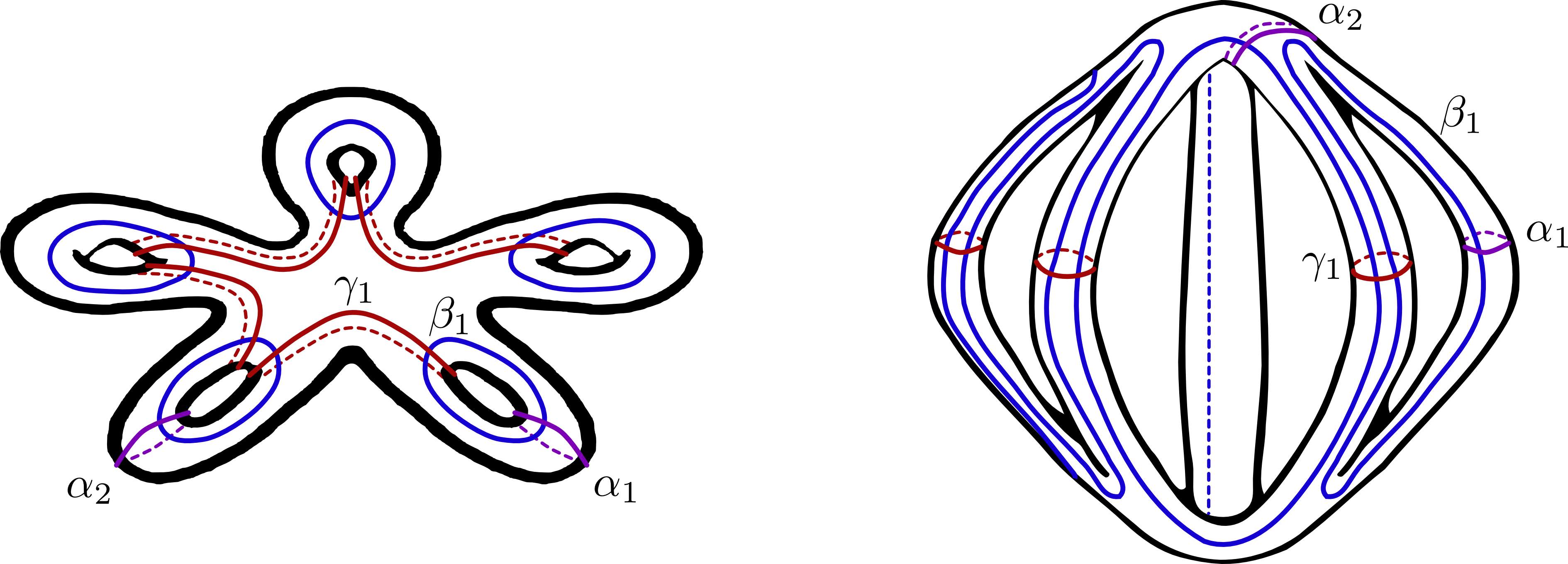}
\caption{The embeddings of chains of curves in the $\sigma_i$. The $\alpha$ curves are only included in the subsurface $\sigma_1$.}
\label{fig:54top}
\end{figure}
Note that the curves in each $F_i$ form a chain of simple closed curves in $S_g$ and the union of the $F_i$ is nonseparating. (Note that $F_1$ is not quite a chain because of the $\alpha_2$ curve.) By the change of coordinates principle, there is a homeomorphism $\hat{f}$ that takes curves in the $F_i$ to the curves in the chains in $\sigma_i$ as specified above, as these chains of curves have the same length. (Recall that in $\sigma_1$ we have two additional curves that correspond to the $\alpha$ curves of $S_g$.) Let $f$ be the mapping class of $\hat{f}^{-1}r\hat{f}$. Then $f$ has order $k$ and maps $\gamma_1$ to $\gamma_2$ as required by Lemma \ref{lemma:four}. 

We now construct $\hat{g}$. We form triples of curves $G_i$, $2 \leq i \leq a+b$. To form each $G_i$, we take the second-to-last $\beta$ curve in $F_{i-1}$, the excluded $\gamma$ curve falling between $F_{i-1}$ and $F_i$, and the second $\beta$ curve in $F_i$.

Note that the curves in $\cup_i G_i$ are in the complement of $L$, that they are nonseparating simple closed curves, that they are disjoint, and that their union is nonseparating. By the change of coordinates principle, there is a homeomorphism $\hat{g}$ that maps the curves in $L$ and the curves $\cup_i G_i$ to a collection of curves of the same topological type in $\Sigma_g$ as follows:
\begin{align*}
\hat{g}\colon  S_g &\longrightarrow \Sigma_g \\
(x_3,x_1,\gamma_1,\gamma_2) &\longmapsto        
 (a,b,c,d) \text{ in } \sigma_1 \text{ as in Figure \ref{fig:5lantern}}  \\
 G_i &\longmapsto (\beta_1,\beta_2,\beta_3) \text{ in } \sigma_i, \ 2\leq i \leq a+b
\end{align*}
Note that the specified image curves are of the same topological type as the four curves in $L$ and the curves in $\cup_i G_i$. Note also that the embedding of the lantern curves depends on whether the genus of $\sigma_1$ is $k$ or $k-1$; see Figure \ref{fig:5lantern}. In both embeddings, the image of $L$ is nonseparating. 
Let $g$ be the mapping class of $\hat{g}^{-1}r\hat{g}$. Then $g$ has order $k$ and maps the pair $(x_3,x_1)$ to the pair $(\gamma_1,\gamma_2)$ as required in Lemma \ref{lemma:four}.

\begin{figure}[ht]
\centering
\includegraphics[width=350px]{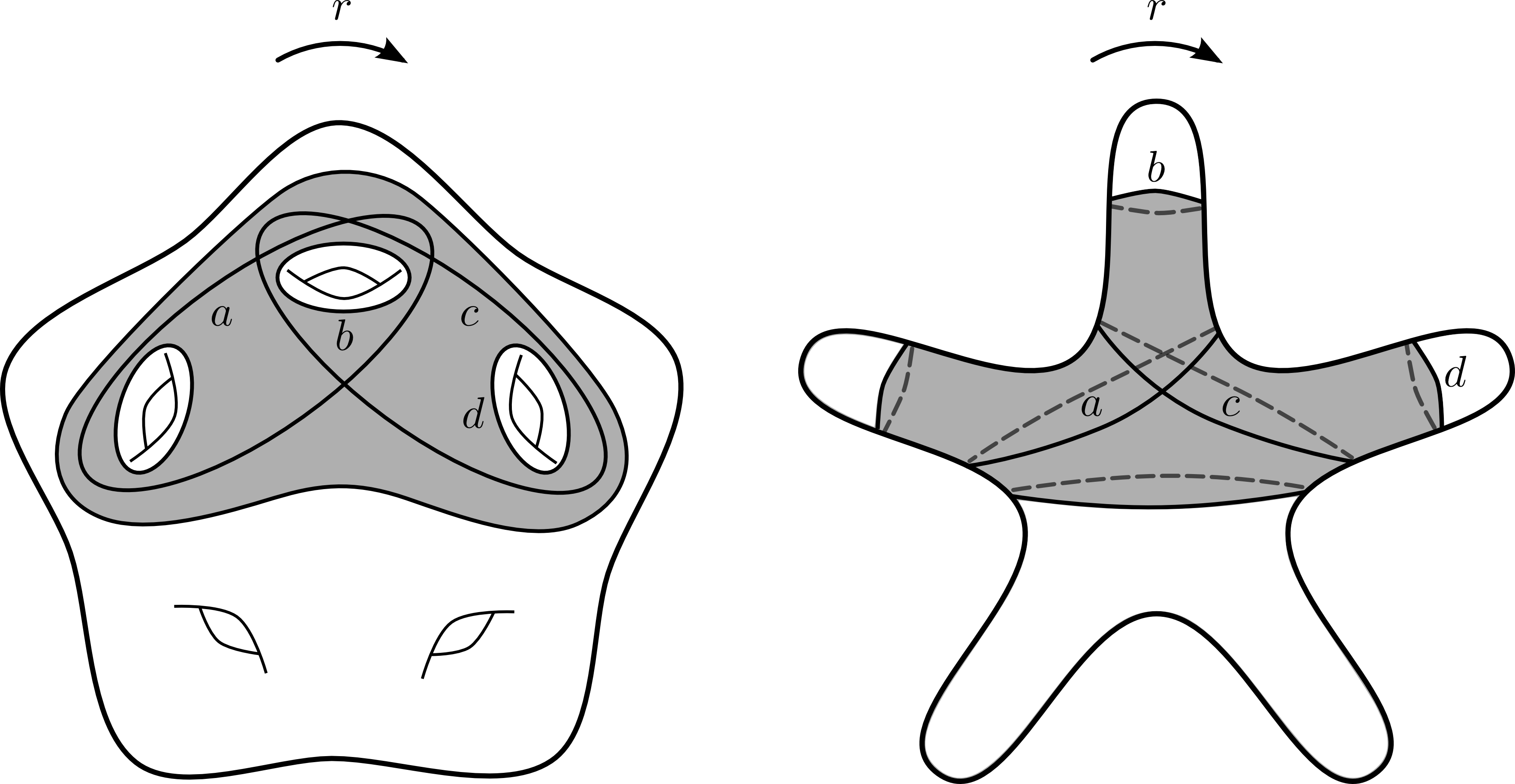}
\caption{The important curves of the subsurface $L$ as embedded by $\hat{g}$ and $\hat{h}$ in $\sigma_1$ when the genus of $\sigma_1$ is $k$ and when it is $k-1$. The latter is depicted as seen from above.}
\label{fig:5lantern}
\end{figure}

\begin{figure}[p]
\centering
\includegraphics[width=\textwidth]{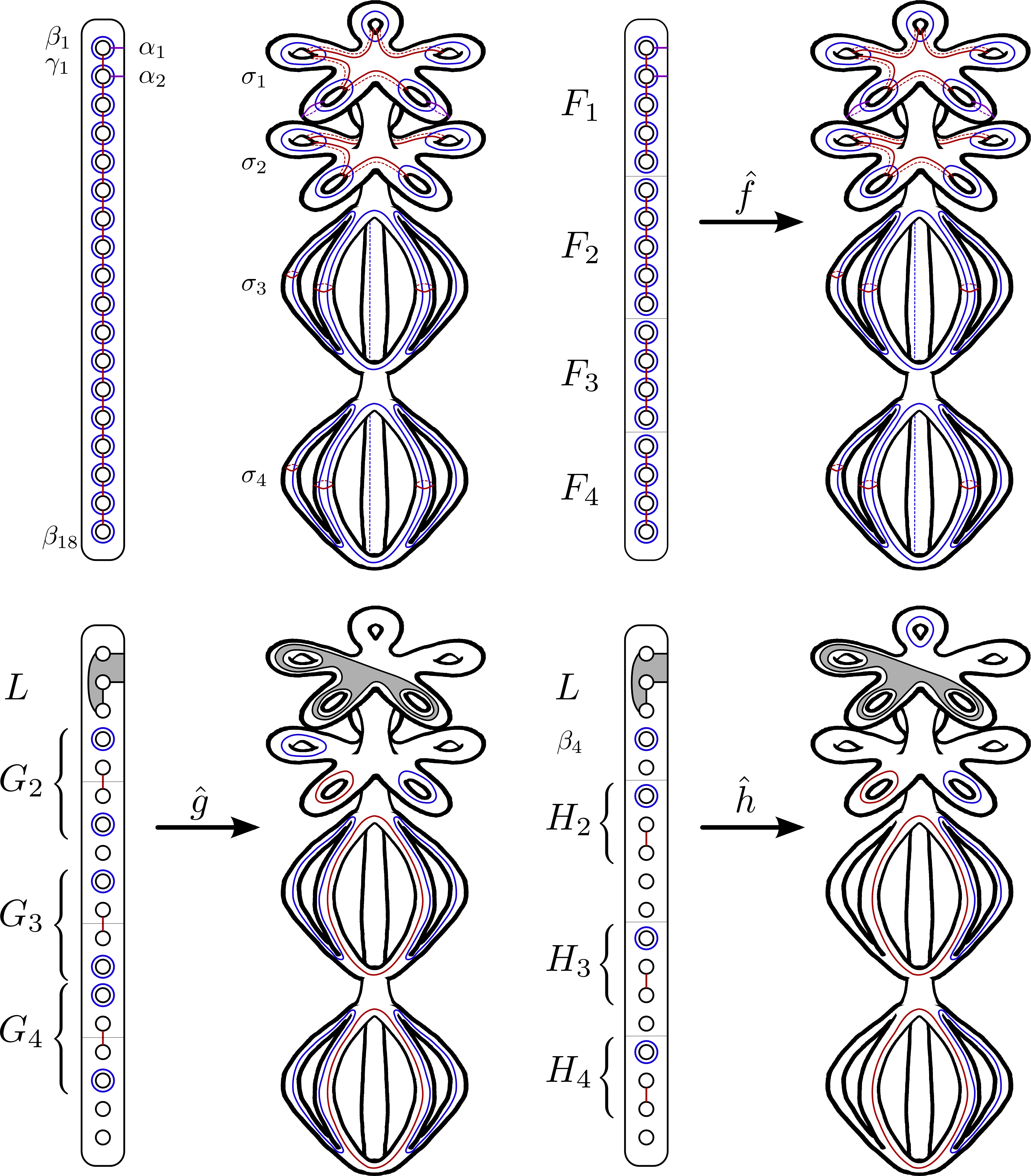}
\caption{The Humphries curves in $S_{18}$. $\Sigma_{18}$ with ``local coordinate" curves in each $\sigma_i$. The curves in the $F_i$, $G_i$, $H_i$, and the subsurface $L$, along with their images under $\hat{f}$, $\hat{g}$, and $\hat{h}$.}
\label{fig:ghat}
\end{figure}

Finally, we construct $\hat{h}$.  We form pairs of curves $H_i$, $2 \leq i \leq a+b$. Each $H_i$ consists of the first $\beta$ curve and the second $\gamma$ curve in $F_i$. The map $\hat{h}$ also specifies the mapping of the Humphries curve $\beta_4$. Let $\hat{h}$ be a homeomorphism that maps curves as follows:
\begin{align*}
\hat{h}\colon  S_g &\longrightarrow \Sigma_g \\
(\gamma_1,\gamma_2,x_2,\alpha_2) &\longmapsto        
 (a,b,c,d) \text{ in } \sigma_1 \text{ as in Figure \ref{fig:5lantern}}\\
 \beta_4 &\longmapsto r(d) \text{ in }\sigma_1\\
 H_i &\longmapsto (\beta_1,\beta_2) \text{ in } \sigma_i, \ 2\leq i \leq a+b
\end{align*}
Let $h$ be the mapping class of $\hat{h}^{-1}r\hat{h}$. Then $h$ has order $k$ and $h^{-1}$ maps the pair $(x_2,\alpha_2)$ to the pair $(\gamma_1,\gamma_2)$ as required by Lemma \ref{lemma:four}. (We use $h^{-1}$ here because we want all of our generators to be conjugate and because the lantern curves are in a fixed cyclic order.)

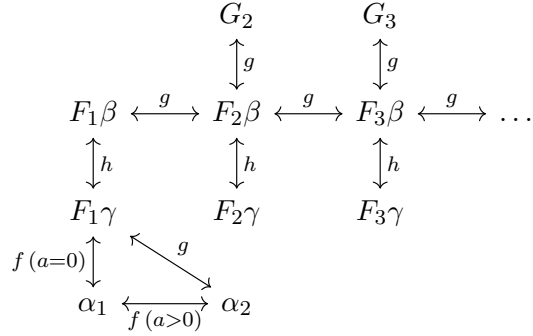
\begin{figure}[ht]
\begin{tikzcd}
& G_2 \arrow[d,leftrightarrow,"g"] & G_3 \arrow[d,leftrightarrow,"g"]\\
F_1\beta \arrow[r,leftrightarrow,"g"] \arrow[d,leftrightarrow,"h"]  & F_2\beta \arrow[r,leftrightarrow,"g"] \arrow[d,leftrightarrow,"h"] & F_3\beta \arrow[r,leftrightarrow,"g"] \arrow[d,leftrightarrow,"h"] & \dots\\
F_1\gamma \arrow[leftrightarrow,"g",to=4-2] \arrow[leftrightarrow,"f\,(a=0)",to=4-1,swap] & F_2\gamma & F_3\gamma\\
\alpha_1 \arrow[r,leftrightarrow,"f\,(a>0)",swap] & \alpha_2
\end{tikzcd} 
\caption{Each node is a collection of curves that are in the same orbit under the subgroup generated by a single element. Each arrow indicates when a power of an element maps a curve in one collection to a curve in another. Since every Humphries curve is in at least one of the collections, all Humphries curves are in the same orbit under the subgroup $\langle f,g,h \rangle$.}
\label{fig:orbits}
\end{figure}

We now show that the Humphries curves are in the same orbit under $\langle f,g,h\rangle$. Refer to Figure \ref{fig:orbits}. First, note that every $\beta$ and $\gamma$ Humphries curve in $S_g$ is in some $F_i$ or $G_i$. Additionally, powers of $f$ map any $\beta$ curve in $F_i$ to any other $\beta$ curve in the same $F_i$, and likewise for $\gamma$ curves. Call these orbits of curves $F_i\beta$ and $F_i\gamma$. In the same way, a power of $g$ maps any curve in $G_i$ to any other curve in the same $G_i$. Thus at most we have the following orbits of the Humphries curves under $\langle f,g,h\rangle$: the $F_i\beta$, the $F_i\gamma$, the $G_i$, $\alpha_1$, and $\alpha_2$. We will show that these are all in fact a single orbit under $\langle f,g,h\rangle$.

The element $f$ maps $\alpha_1$ to $\alpha_2$ when $\sigma_1$ has genus $k$ and maps $\alpha_1$ to $\gamma_1$ when $\sigma_1$ has genus $k-1$. The element $g$ maps the lantern curve $\gamma_2$ to the lantern curve $\alpha_2$. Thus each of $\alpha_1$ and $\alpha_2$ is in the same orbit as some $\gamma$ curve.

A power of $g$ takes a curve in $F_{i}\beta$ to a curve in $F_{i-1}\beta$ as well as to a curve in $G_i$, $2\leq i \leq a+b$. Additionally, a power of $h$ takes a curve in $F_i\beta$ to a curve in $F_i\gamma$, $1 \leq i \leq a+b$. (Note that in the case of $F_1$, we have $h^2(\gamma_2)=\beta_4$.) Thus all Humphries curves are in a single orbit under $\langle f,g,h \rangle$. By Lemma \ref{lemma:four}, the Dehn twist about $\alpha_1$ may be written as a product in $f$, $g$, $h$, and $T_{\gamma_2}fT^{-1}_{\gamma_2}$. Thus all Dehn twists about the Humphries curves may be written as products in our four elements of order $k$, and so they generate $\Mod(S_g)$.

In the case where $g=ak+1$, we may modify the construction to show that $\Mod(S_g)$ is again generated by four elements of order $k$. Take a connect sum of $a$ surfaces of genus $k$ and insert one further handle along the axis of rotation, as in Lemma \ref{thm:mainlemma}. The element $r$ is a rotation of this embedded surface by $2\pi/k$ and $f$ is defined as above by ignoring the final two Humphries curves $\beta_g$ and $\gamma_{g-1}$. We must modify our other elements of order $k$ so that they place these two additional Humphries curves into the same orbit as all of the other Humphries curves under the subgroup $\langle f,g,h \rangle$. Modify $\hat{g}$ so that it additionally maps $\beta_g$ to $r(d)$ in $\sigma_1$ and modify $\hat{h}$ so that it additionally maps $\gamma_{g-1}$ to $r^2(d)$ in $\sigma_1$. These modifications preserve the fact that the curves involved are disjoint and that their union is nonseparating. The elements $g$ and $h$ now put the curves $\beta_g$ and $\gamma_{g-1}$ into the same orbit as the other Humphries curves. Hence $\Mod(S_g)$ is also generated by four elements of order $k$ when $g=ak+1$.
\end{proof}

\section{Sharpening to three elements}
We now prove the second part of Theorem \ref{thm:maintheorem}.

\begin{proof}We first provide the construction for the cases $k \geq 8$, and then afterwards give the constructions for $k=7$ and $k=6$. Let $k \geq 8$. By assumption we may write $g$ in the form $ak+b(k-1)$. We construct the homeomorphism $\hat{f}:S_g\rightarrow \Sigma_g$ as in the proof of the first part of theorem, except with the modification that it additionally maps the $\alpha$ curve that intersects the final $\beta$ curve in $F_1$ (called $\alpha_\ell$) to the curve $r^{-1}\hat{f}(\alpha_1)$. See Figure \ref{fig:alphas}. We again let $f$ be the mapping class of $\hat{f}^{-1}r\hat{f}$, and $f$ has order $k$.

We now construct $\hat{g}$. Let $G_2$ consist of $\alpha_\ell$, the excluded $\gamma$ curve falling between $F_1$ and $F_2$, the first $\gamma$ curve in $F_2$, and the third $\beta$ curve in $F_2$. For $2 < i \leq a+b$, let $G_i$ be the last $\gamma$ curve in $F_{i-1}$, the excluded $\gamma$ curve between $F_{i-1}$ and $F_i$, the first $\gamma$ curve in $F_i$, and the third $\beta$ curve in $F_i$. See Figure \ref{fig:worst}. Let $\hat{g}$ be a homeomorphism that maps the specified curves as follows:

\begin{align*}
\hat{g}\colon  S_g &\longrightarrow \Sigma_g \\
(x_3,x_1,\gamma_1,\gamma_2,x_2,\alpha_2) &\longmapsto (a,b,c,d,e,f) \text{ as in Figure \ref{fig:8}}\\
(\gamma_3,\gamma_4) &\longmapsto (r^3(e),r^3(f))\\
\beta_6 &\longmapsto r^4(f) \text{ in }\sigma_1\\
G_i &\longmapsto (\beta_1,\beta_2,\beta_3,\beta_4) \text{ in } \sigma_i, \ 2\leq i \leq a+b
\end{align*}

Let $g$ be the mapping class of $\hat{g}^{-1}r\hat{g}$. Then $g$ has order $k$ and maps the pair $(x_3,x_1)$ to the pair $(\gamma_1,\gamma_2)$ as required by Lemma \ref{lemma:four}. Additionally, $g^3(x_2,\alpha_2)=(\gamma_3,\gamma_4)$ and $f^{-2}(\gamma_3,\gamma_4)=(\gamma_1,\gamma_2)$. Therefore we may define $h=f^{-2}g^3$ so that $h$ satisfies Lemma \ref{lemma:four} because $h(x_2,\alpha_2)=(\gamma_1,\gamma_2)$. Thus the Dehn twist about $\alpha_1$ may be written as a product in $f$, $g$, and $T_{\gamma_2}fT^{-1}_{\gamma_2}$.

\begin{figure}[p]
\centering
\includegraphics[width=320pt]{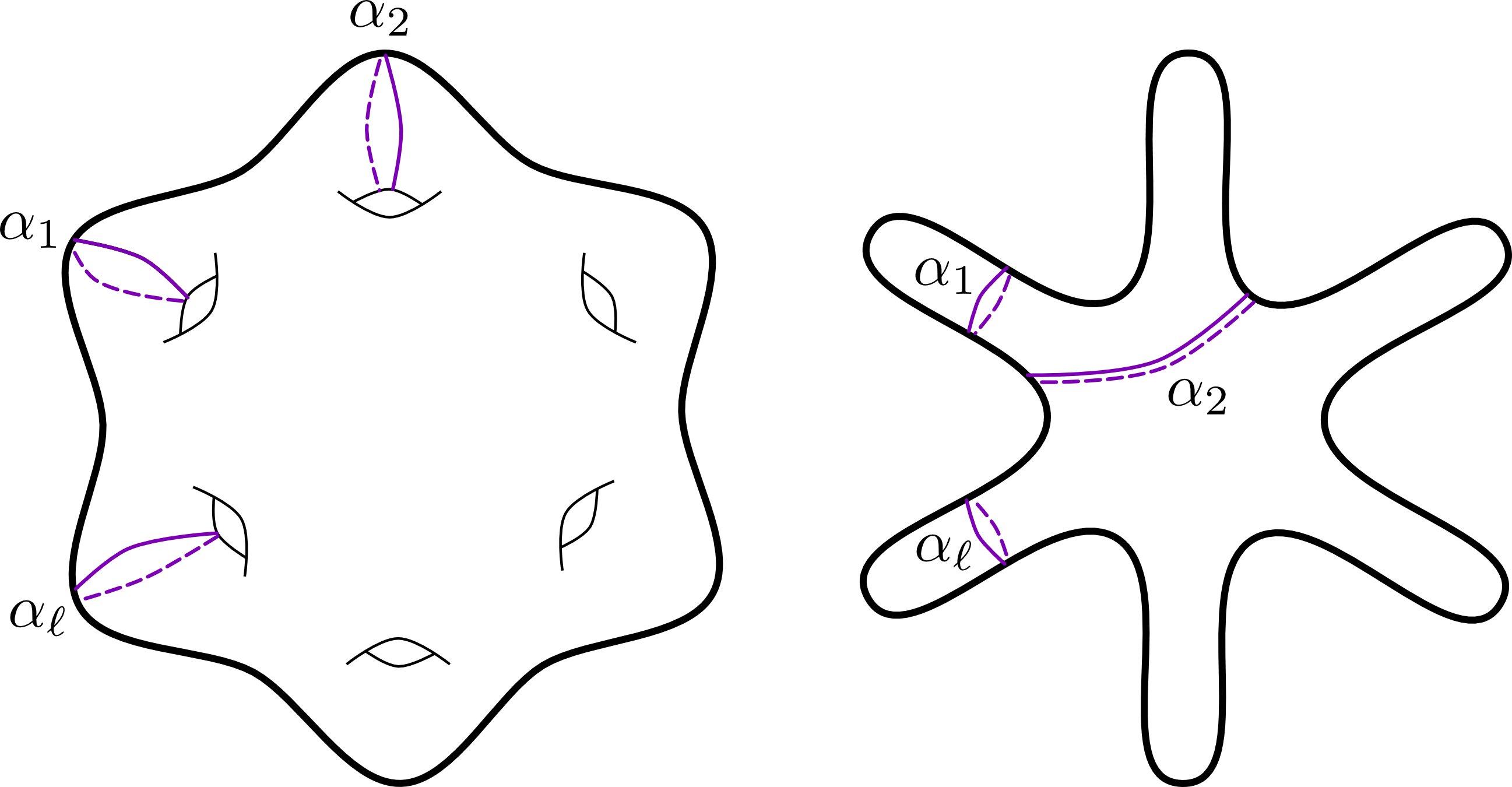}
\caption{The images of $\alpha$ curves embedded in $\sigma_1$ by $\hat{f}$. With this embedding, the rotation $r$ maps $\alpha_\ell$ to $\alpha_1$.}
\label{fig:alphas}
\vspace{10mm}
\centering
\includegraphics[width=320pt]{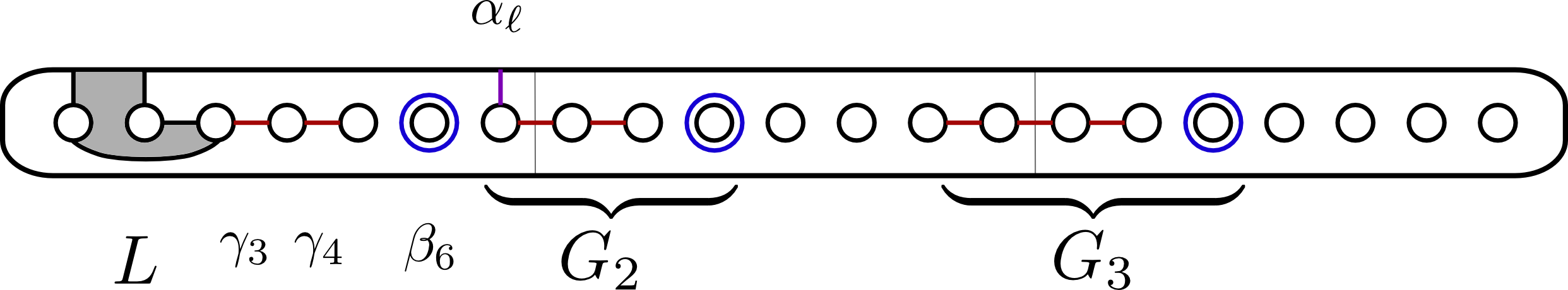}
\caption{The curves mapped by $\hat{g}$ in the case $k=8$, $g=21$. This is a worst case example where $k$ has the smallest possible value and all of the $\sigma_i$ have genus $k-1$.}
\label{fig:worst}
\vspace{10mm}
\centering
\includegraphics[width=400pt]{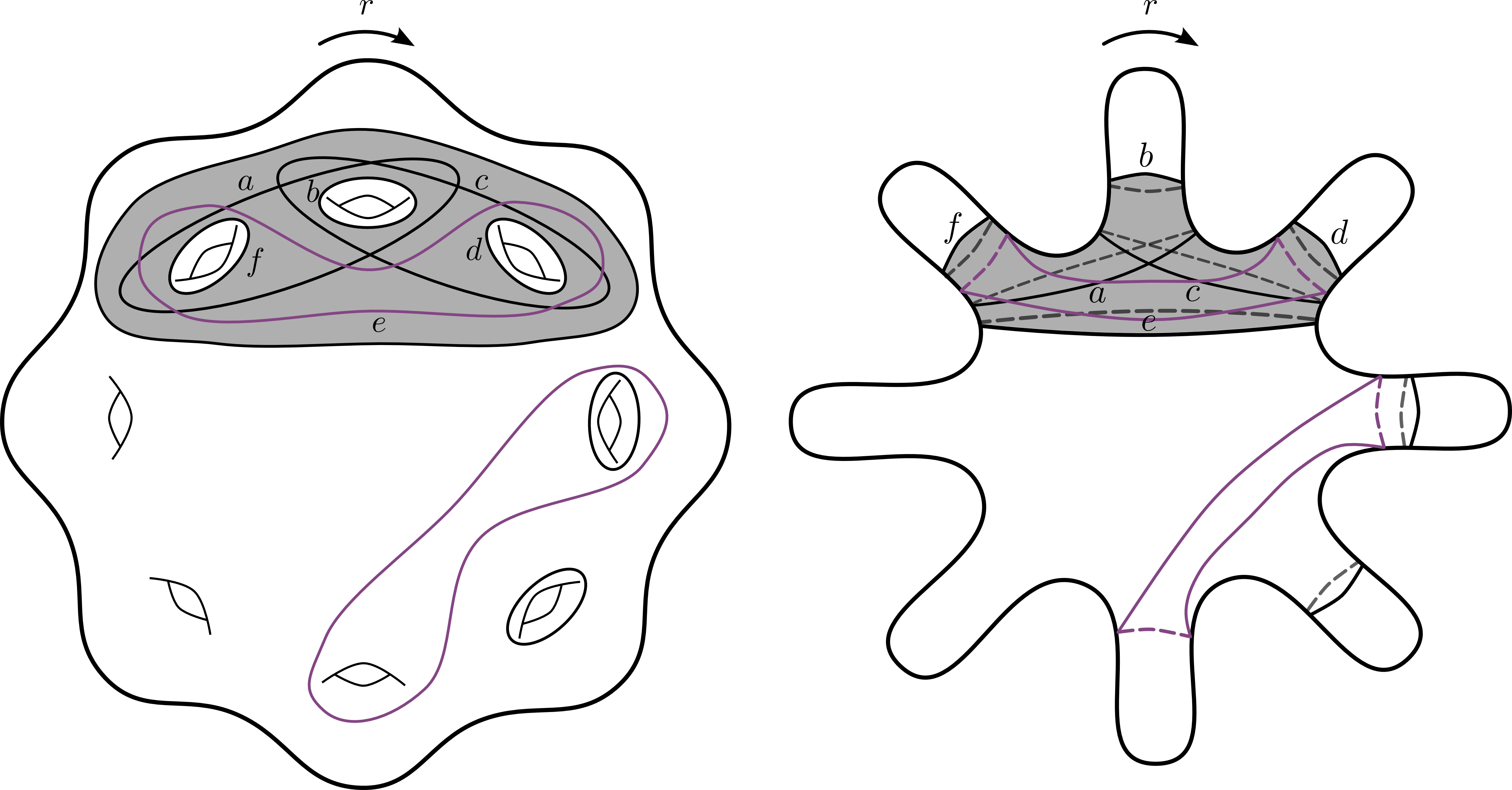}
\caption{The embedding of the subsurface $L$ in $\sigma_1$ when the genus of $\sigma_1$ is $k$ and when it is $k-1$. Also, the embedding of $\gamma_3$, $\gamma_4$, and $\beta_6$. These diagrams depict the case $k=8$.}
\label{fig:8}
\end{figure}

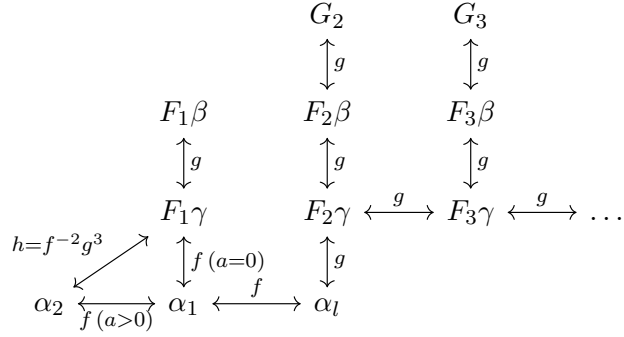
\begin{figure}[ht]
\begin{tikzcd}
& & G_2 \arrow[d,leftrightarrow,"g"] & G_3 \arrow[d,leftrightarrow,"g"]\\
& F_1\beta \arrow[d,leftrightarrow,"g"]  & F_2\beta \arrow[d,leftrightarrow,"g"] & F_3\beta \arrow[d,leftrightarrow,"g"]\\
& F_1\gamma \arrow[d,leftrightarrow,"f\,(a=0)"] \arrow[leftrightarrow,"h=f^{-2}g^3",to=4-1,swap] & F_2\gamma \arrow[r,leftrightarrow,"g"] \arrow[d,leftrightarrow,"g"] & F_3\gamma \arrow[r,leftrightarrow,"g"] & \dots\\
\alpha_2 \arrow[r,leftrightarrow,"f\,(a>0)",swap] & \alpha_1 \arrow[r,leftrightarrow,"f"] & \alpha_{l}\\
\end{tikzcd} 
\caption{Again, each node is a collection of curves that are in the same orbit under the subgroup generated by a single element. Each arrow indicates when a power of an element maps a curve in one collection to a curve in another. Since every Humphries curve is in at least one of the collections, all Humphries curves are in the same orbit under $\langle f,g \rangle$.}
\label{fig:orbits3}
\end{figure}

Finally, we show that all of the Humphries curves are in the same orbit under $\langle f,g \rangle$, as can be seen in Figure \ref{fig:orbits3}. Again, every $\beta$ and $\gamma$ Humphries curve in $S_g$ is in some $F_i\beta$, $F_i\gamma$, or $G_i$. Therefore we have at most the following orbits of the Humphries curves under $\langle f,g\rangle$: the $F_i\beta$, the $F_i\gamma$, the $G_i$, $\alpha_1$, and $\alpha_2$. We will show that these are all in fact a single orbit under $\langle f,g\rangle$. For $i>2$, powers of $g$ put the curves in $G_i$, $F_i\beta$, $F_i\gamma$, and $F_{i-1}\gamma$ in the same orbit. Note that powers of $g$ map $\beta_6$ to $\gamma_2$ and a $\gamma$ curve in $F_2$ to $\alpha_\ell$, while the product $h$ carries $\alpha_2$ to $\gamma_2$. Also, $f$ maps $\alpha_\ell$ to $\alpha_1$ and maps $\alpha_1$ either to $\alpha_2$ or $\gamma_1$, depending on the genus of $\sigma_1$. Considering this, all of the curves are in the same orbit under the subgroup $\langle f,g\rangle$. Therefore the Dehn twist about each of the Humphries curves may be written as a product in the three elements $f$, $g$, and $T_{\gamma_2}fT^{-1}_{\gamma_2}$, and so they generate $\Mod(S_g)$.

In the case where $k=7$, the same construction as above goes through as long as the genus of $\sigma_1$ is 7. As illustrated in Figure \ref{fig:worst}, under this assumption there is enough room to configure all of the required curves in the construction of $\hat{g}$. The hypotheses of the theorem in this case exactly demand that the genus of $\sigma_1$ be 7.

In the case where $k=6$, we use the same construction as above for $f$ and use the following alternative construction for the element $g$ that takes advantage of the three-fold symmetry of a lantern. See Figure \ref{fig:lantern6}. Let $\hat{g}$ be a homeomorphism that maps the specified curves as follows:
\begin{align*}
\hat{g}\colon  S_g &\longrightarrow \Sigma_g \\
(x_3,x_1,\gamma_1,\gamma_2,x_2,\alpha_2) &\longmapsto (a,b,r^2(a),r^2(b),r^{4}(a),r^{4}(b)) \text{ as in Figure \ref{fig:lantern6}}\\
\beta_4 &\longmapsto r(b) \text{ in }\sigma_1\\
G_i &\longmapsto (\beta_1,\beta_2,\beta_3,\beta_4) \text{ in } \sigma_i, \ 2\leq i \leq a+b
\end{align*}

In this case, $g^2$ and $g^4$ play the roles of $g$ and $h$ in Lemma \ref{lemma:four}, and all Humphries curves are again in the same orbit.
\end{proof}
\begin{figure}[ht]
\centering
\includegraphics[width=320pt]{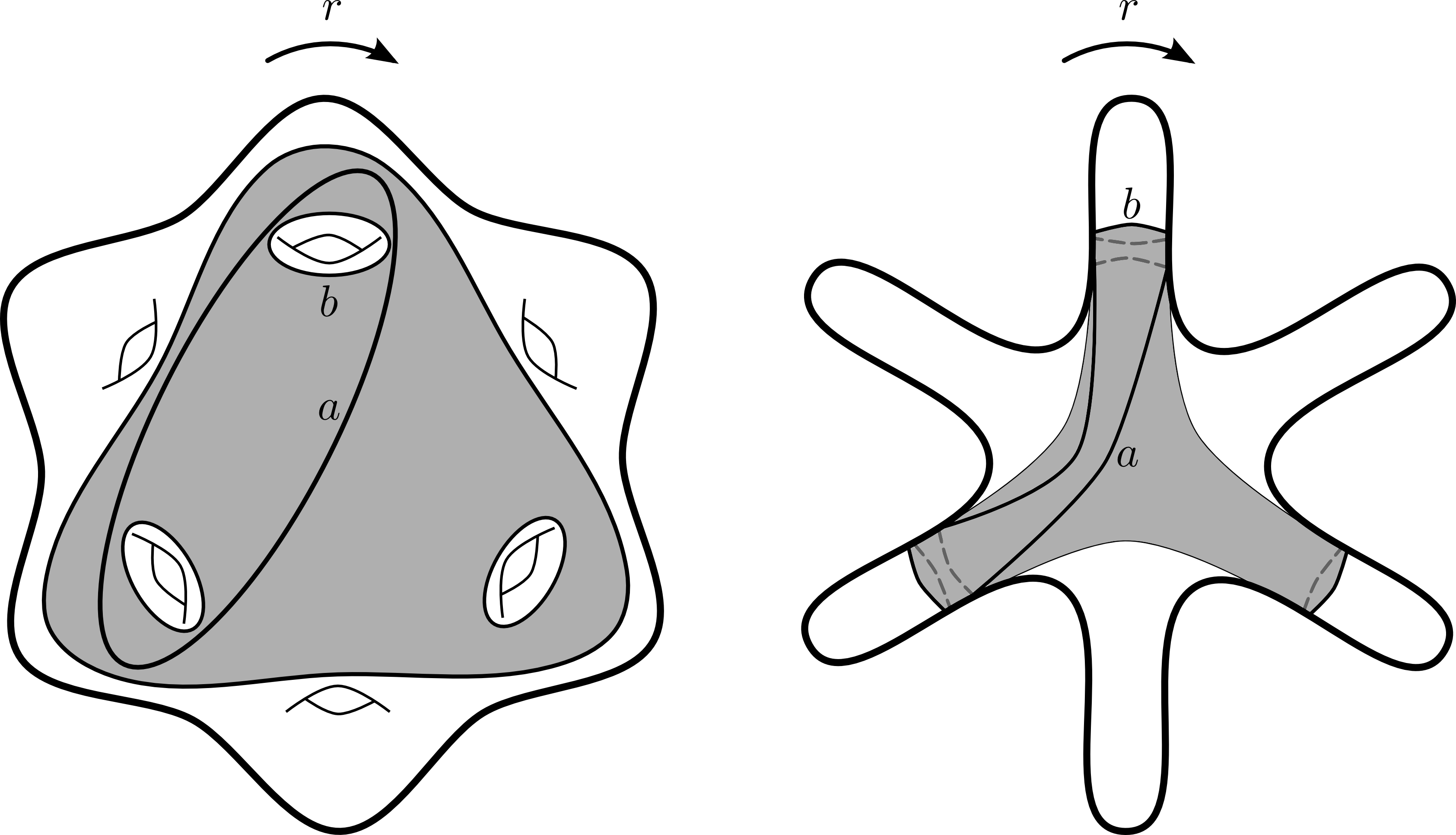}
\caption{The embeddings of the subsurface $L$ in $\sigma_1$ when $k=6$.}
\label{fig:lantern6}
\end{figure}

\section{The symmetric and alternating groups}
In this section we show that $\Sigma_n$ and $\A_n$ can be generated by a uniformly small number of elements of fixed order $k$. Note that since all permutations of odd order are even permutations, elements of odd order $k$ cannot generate $\Sigma_n$. For the case $k=2$, two elements of order 2 cannot generate $\Sigma_n$ or $\A_n$ except for small values of $n$, since any group generated by two involutions is a quotient of a dihedral group. Therefore the result by Nuzhin \cite{Nuzhin} that three elements of order 2 generate $\A_n$ for $n=5$ and $n\geq 9$ is the best possible in general.

We take up the cases where $k \geq 3$. We break up Theorem \ref{thm:perm} into the following two propositions:

\begin{proposition}
Let $k \geq 3$ and $n \geq k$. Then three elements of order $k$ suffice to generate $\Sigma_n$ when $k$ is even and to generate $\A_n$ when $k$ is odd.
\label{thm:symmetric}
\end{proposition}

\begin{proposition}
Let $k \geq 3$ and $n \geq k+2$. Then four elements of order $k$ suffice to generate $\A_n$ when $k$ is even.
\label{thm:alternating}
\end{proposition}

The bound in Proposition \ref{thm:alternating} is $n \geq k+2$, and this is different from the bound in Proposition \ref{thm:symmetric}. In the case when $k$ is even and $k=n-1$, $\A_n$ is not always generated by elements of order $k$. For instance, if $k=n-1$ is a power of 2, then the only elements of order $k$ are the $k$-cycles, but these are odd permutations and so cannot generate $\A_n$.

It is possible that two elements of order $k$ may in fact suffice for $k \geq 3$. In this direction, Miller \cite{Miller2} shows that if \mbox{$2 \leq k \leq n \leq 2k-1$}, two $k$-cycles suffice to generate $\Sigma_n$ when $k$ is even and $\A_n$ when $k$ is odd. In Section 8 we give a construction for a pair of elements of order $k$ that have generated $\Sigma_n$ for $k$ even and $\A_n$ for $k$ odd for all values $(k,n)$ that we have tested by computer calculation, excepting a few small values that can be handled separately.

\subsection*{Preliminaries}
We take $N=\{0, \ \dots, \ n-1\}$ as our underlying permuted set. Denote by $h_{k,n}(a)$ a \textit{step $k$-cycle}, which is a $k$-cycle of the form \mbox{$(a \ a+1 \ \cdots \ a+k-1)$} with entries taken mod $n$. We further define $s_{k,n}(a,\ell)$ to be a \textit{sequential step product} so that

$$s_{k,n}(a,\ell)=\prod_{i=1}^{\ell} h_{k,n}(a+ (i-1) \cdot k)$$
with entries taken mod $n$. By way of example, we have 

$$s_{4,15}(6,3)=(6 \ 7 \ 8 \ 9)(10 \ 11 \ 12 \ 13)(14 \ 0 \ 1 \ 2).$$
Note that in order to obtain a product of disjoint cycles, the largest value that $\ell$ may take is $\left\lfloor n/k\right\rfloor$.

The main result about permutation groups that we use in our proofs is Jordan's theorem. Recall that a permutation group $G$ is \textit{transitive} if it acts transitively on the underlying permuted set, and it is \textit{2-transitive} if it acts transitively on ordered pairs of distinct elements of the underlying permuted set. A permutation group $G$ is \textit{primitive} if it is transitive and if no nontrivial partition of the underlying permuted set is preserved by the action of $G$.

\begin{theorem*}[Jordan]
Let $G$ be a primitive subgroup of $\Sigma_n$, and suppose $G$ contains a $p$-cycle where $p$ is prime and $p \leq n-3$. Then $G$ is either $\A_n$ or $\Sigma_n$.
\end{theorem*}
For additional background on primitivity and Jordan's theorem, see for instance the book of Isaacs \cite[Chapter~8B]{Isaacs}.

\begin{proof}[Proof of Proposition \ref{thm:symmetric}]
Miller \cite{Miller2} showed that for $n \leq 2k-1$, two $k$-cycles generate $\Sigma_n$ when $k$ is even and $\A_n$ when $k$ is odd. Thus we may assume that $n \geq 2k$. Consider the permutation group $G$ on the set $N$ generated by the following elements:
\begin{align*}
    a &=s_{k,n}(0,\left\lfloor n/k\right\rfloor)\\
    b &=\begin{cases}
    s_{k,n}(k-1,\left\lfloor n/k\right\rfloor), &\text{if $k \nmid n$}\\
    s_{k,n}(k-1,\left\lfloor n/k\right\rfloor-1), &\text{if $k | n$}
  \end{cases} \\
    c &=\begin{cases}
    (0 \ 1 \ 2), & \text{if $k=3$}\\
    (0 \ 1 \ 2) \ h_{k,n}(0)=(1 \ 0 \ 2 \ \cdots \ k-1), & \text{if $k>3$}
    \end{cases}
\end{align*}
All three elements are products of disjoint $k$-cycles and so have order $k$. As an illustration, here are the three elements in the case $k=5$, $n=18$.

 \[\setlength{\arraycolsep}{4pt}
  \begin{matrix}
  a & = &  \llap{(}0& 1 & 2 & 3 & 4 \rlap{)}&\llap{(}5 & 6 & 7 & 8 & 9\rlap{)}&\llap{(}10 & 11 & 12 & 13 & 14 \rlap{)}\\
  b & = &   &  &  &  & \llap{(}4 & 5&6 & 7 & 8 \rlap{)}&\llap{(}9 & 10 & 11 & 12 & 13\rlap{)}& \llap{(}14 & 15 & 16 & 17 & 0\rlap{)}\\
  c & = & \llap{(}1 & 0 & 2 & 3 & 4\rlap{)}
  \end{matrix}
\]

To apply Jordan's theorem, we must show that $G$ contains a small prime cycle and that $G$ is primitive. Since $n \geq 2k \geq 6$, a 3-cycle will satisfy the small prime cycle requirement in Jordan's theorem. If $k=3$, then $c$ is a 3-cycle. If $k>3$, the commutator $[a,c]$ is the 3-cycle $(0 \ 1 \ k-2)$.

It remains to show that $G$ is primitive, and it suffices to prove the stronger condition that $G$ is 2-transitive. We thus aim to show that for an arbitrary ordered pair $(i,j)\in N^2$, $i \neq j$, there exists $g \in G$ such that $g(i,j) = (k-2,k-1)$. $G$ is certainly transitive, since the overlapping cycles of $a$ and $b$ allow any element in $N$ to be carried to any other. Let $g_1$ be a product in $a$ and $b$ such that $g_1(i)=k-2$. We now seek a $g_2$ that carries $j$ to $k-1$ while keeping $i$ at $k-2$. We reduce to the case where $g_1(j)$ sits outside of $S=\{0, \dots, k-1\}$. If $g_1(j)$ sits in $S$, then we may first move $g_1(j)$ outside of $S$ while keeping $i$ at $k-2$. Either $b$ acts on $g_1(j)$ while keeping $i$ at $k-2$ (and so can move $g_1(j)$ out of $S$), or else $c^mbc^{-m}$ does so for some power of $c$. Therefore $i$ sits at $k-2$ and $g_1(j)$ sits outside of $S$. Note that $b$ and the product $c^{-1}a$ each fix $k-2$. Since $b$ and $c^{-1}a$ act transitively on $N-\{0, \dots, k-2\}$, we may form a product $g_2$ in $b$ and $c^{-1}a$ so that $g_2g_1(i,j)=(k-2,k-1)$. Thus $G$ is 2-transitive, and so it is primitive.

Applying Jordan's theorem, we have that $G$ is either $\A_n$ or $\Sigma_n$. If $k$ is even, then $G$ contains the odd permutation $c$, and therefore $G \cong \Sigma_n$. If $k$ is odd, then all of the generators of $G$ are even permutations, and so $G \cong \A_n$. \end{proof}

\begin{proof}[Proof of Proposition \ref{thm:alternating}]
Take $k$ to be even and $n \geq k+2$. To show that $\A_n$ is generated by at most four elements of order $k$, we will modify the generating set for $\Sigma_{n-2}$ comprised of three elements of order $k$ from the proof of Proposition \ref{thm:symmetric} or the two $k$-cycles given by Miller. First, add elements $\{a,b\}$ to the underlying set of permuted objects $\{0, \ \dots, \ n-3\}$. For each odd permutation in the generating set for $\Sigma_{n-2}$, multiply it by the transposition $(a \ b)$ so that it becomes an even permutation. For each even permutation of the generating set for $\Sigma_{n-2}$, let it fix $a$ and $b$ so that it remains an even permutation. Finally, add to the generating set the element $t=(a \ b \ 3 \ 4 \ \cdots \ k)(1 \ 2)$. This is an even permutation of order $k$.

These elements together generate $\A_n$, since every generator is an even permutation and every 3-cycle on $\{0, \ \dots, \ n-3, \ a, \ b \}$ is generated by them. To see this last fact, observe that the 3-cycles on $\{0, \ \dots, \ n-3\}$ are generated by the modified elements and also that any 3-cycle involving $a$ or $b$ is a conjugation of one of these 3-cycles by a power of $t$. Therefore we have a generating set for $\A_n$ comprised of at most four elements of even order $k$.\end{proof}

\section{Results for automorphism groups of free groups and for linear groups}

In this section we use our results for $\A_n$ and $\Mod(S_g)$ to derive similar results for $\Aut^+(F_n)$, $\Out^+(F_n)$, $\SL(n,\mathbb{Z})$, and $\Sp(2n,\mathbb{Z})$. For the first three families of groups, the numbers of elements of order $k$ required in each case are simply the sums of the number of elements of order $k$ needed to generate $\A_n$ and $\Mod(S_g)$ individually. For $\Sp(2n,\mathbb{Z})$, the number of elements of order $k$ needed is the same as for $\Mod(S_g)$, $g=n$.

\begin{proof}[Proof of Theorem \ref{thm:others}]
Fix $k \geq 5$. As mentioned above, Gersten \cite{Gersten} gave a presentation for $\Aut^+({F_n})$ and in particular showed that $\Aut^+({F_n})$ is generated by the collection of left and right transvections of one free group generator by another. It is immediate that $\Aut^+({F_n})$ is generated by a single left transvection, a single right transvection, and a collection of elements that act 2-transitively on the generators of $F_n$. We must therefore write such elements as products in eight elements of order $k$ whenever $n \geq 2(k-1)$.

First, there is a natural inclusion map from $\A_n$ to $\Aut^+(F_n)$ where
a permutation maps to the automorphism that effects that permutation on
the generators of $F_n$. Through this inclusion, $\A_n$ gives a 2-transitive action on the generators $F_n$. By Theorem \ref{thm:perm}, $\A_n$ may be generated by at most four elements of order $k$ for $n \geq k+2$. Therefore there are four elements of order $k$ in $\Aut^+(F_n)$ that effect a 2-transitive action on the generators of $F_n$ for $n \geq 2(k-1)$.

Next, we show that the generating set for $\Mod(S_{k-1})$ constructed in the proof of Theorem \ref{thm:maintheorem} also gives a generating set for $\Mod(S_{k-1,1})$, where $S_{k-1,1}$ denotes a surface of genus $k-1$ with a single puncture. The homeomorphisms of $S_{k-1}$ of order $k$ produced in our proof of Theorem \ref{thm:maintheorem} have two fixed points $p_1$ and $p_2$, since in fact the rotation $r$ fixes four points where the axis of rotation intersects the surface. Therefore we may place a puncture at $p_1$ so that it is fixed by all homeomorphisms produced in the proof of Theorem \ref{thm:maintheorem}, and therefore also by all corresponding mapping classes.

By taking the other fixed point $p_2$ as a base point, we may embed in $S_{k-1,1}$ a system of generators $a_1,\dots,a_{2(k-1)}$ for $\pi_1(S_{k-1,1},p_2)\cong F_{2(k-1)}$. We therefore have that
$$\Mod(S_{k-1,1},p_2) \xhookrightarrow{i} \Aut^+(\pi_1(S_{k-1,1})) \cong \Aut^+(F_{2(k-1)}).$$
The image of a Dehn twist under this inclusion is a transvection. By a standard construction, specific Dehn twists $T_{c_1}$ and $T_{c_2}$ are mapped by $i$ to the left transvection $t_1:a_1 \rightarrow a_2a_1$ and the right transvection $t_2:a_3 \rightarrow a_3a_4$, where each automorphism $t_i$ fixes all other generators of the free group.

We can take images of the transvections $t_i$ under the following inclusions.
$$\Aut^+(F_{2(k-1)}) \xhookrightarrow{} \Aut(F_{2(k-1)}) \xhookrightarrow{} \Aut(F_{2(k-1)+1}) \xhookrightarrow{} \Aut(F_{2(k-1)+2}) \xhookrightarrow{} \dots$$
In particular, these images are still a left and a right transvection by a generator of the free group. They are also elements of the corresponding special automorphism groups, since their images under the surjection to the corresponding general linear group also have determinant 1.

Therefore the inclusions of the (at most) four generators of $\A_n$ and the four generators of $\Mod(S_{k-1,1})$ together generate $\Aut^+(F_n)$ for $n \geq 2(k-1)$, and these are all of order $k$. Finally, when $k \geq 6$, recall that only three elements of order $k$ are needed to generate $\Mod(S_{k})$, and these elements also all contain two fixed points. So when $k \geq 6$ and $n \geq 2k$, seven elements of order $k$ suffice to generate $\Aut^+(F_n)$.

Since $\Aut^+(F_n)$ maps onto $\Out^+(F_n)$ and the kernel is torsion-free, the same number of elements of order $k$ under the same conditions generate $\Out^+(F_n)$. Additionally, since $\Out^+(F_n)$ maps onto $\SL(n,\mathbb{Z})$ \cite{Nielsen} and the kernel is torsion-free \cite{BT}, we also have that the same number of elements of order $k$ under the same conditions generate $\SL(n,\mathbb{Z})$.\end{proof}

\begin{proof}[Proof of Theorem \ref{thm:Sp}]
We have that $\Mod(S_g)$ maps onto $\Sp(2g,\mathbb{Z})$ \cite[Theorem~6.4]{Primer} and that the kernel is torsion-free \cite[Theorem~6.8]{Primer}. Therefore the images of the elements of order $k$ that generate $\Mod(S_g)$ also have order $k$ and generate $\Sp(2g,\mathbb{Z})$.
\end{proof}

\section{Further questions}
The proof of Lemma \ref{thm:mainlemma} provides concrete descriptions of some periodic elements that exist in $\Mod(S_g)$. It also provides an upper bound of $(k-1)(k-3)-1$ on the largest $g$ for which elements of order $k$ fail to exist in $\Mod(S_g)$. We would like to highlight a problem that has already received some attention in the literature.

\begin{problem} Give a formula in $k$ for the largest $g$ where elements of order $k$ fail to exist in $\Mod(S_g)$.
\end{problem}
Such a formula would provide a nice bookend to a formula derived by Harvey \cite{Harvey} that specifies the smallest $g$ where $\Mod(S_g)$ contains an element of order $k$. Problem 1 was solved for elements of prime order $p$ by Harvey \cite[Corollary 13]{Harveyprime} and by Glover and Mislin \cite[Lemma~3.3]{GM}. The formula in this case is $(p^2-4p+1)/2$, $p > 3$. Kulkarni and Maclachlan \cite{KulMac} solved Problem 1 for $k$ a prime power. O'Sullivan and Weaver \cite{OW} call this the largest non-genus problem and give bounds (and in some cases exact formulas) when $k$ is a product of two distinct odd primes. This problem is also sometimes called the stable upper genus problem in the literature.

Next, there are values of $g$ and $k$ that are not covered by our constructions but where elements of order $k$ exist in $\Mod(S_g)$. For instance, there are elements of order 7 in $\Mod(S_3)$.
\begin{problem} Extend Theorem \ref{thm:maintheorem} to all cases where elements of order $k$ exist in $\Mod(S_g)$.
\end{problem}

We can also seek smaller generating sets for $\Mod(S_g)$ consisting of elements of order $k$. We note that any sharpening Theorem \ref{thm:maintheorem} in terms of the number of elements required would seem to demand a new approach, due to the limited symmetries of a lantern.

\begin{question} For any fixed $k \geq 3$ and $g$ sufficiently large, can $\Mod(S_g)$ be generated by two elements of order $k$? What about three elements for orders 4 and 5?
\end{question}

We may of course ask the corresponding sharpening questions for the other groups we have considered. In particular, we would be glad to see the following conjecture resolved.

\begin{conjecture}
Let $k \geq 3$ and $n \geq k$. Then two elements of order $k$ suffice to generate $\Sigma_n$ when $k$ is even and to generate $\A_n$ when $k$ is odd.
\end{conjecture}

We give here a candidate construction for resolving this conjecture. Consider the following pair of permutations on $N$ of order $k$.
\begin{align*}
    a &=s_{k,n}(0,\left\lfloor n/k\right\rfloor),\\
    b &=\begin{cases}
    (k-1 \ k \ k+1) s_{k,n}(k-1,\left\lfloor n/k\right\rfloor), &\text{if $k$ is odd, or $k$ is even and $\left\lfloor n/k\right\rfloor$ is odd}\\
    s_{k,n}(k\left\lfloor n/k\right\rfloor-1,\left\lfloor n/k\right\rfloor-1), &\text{if $k$ is even, $\left\lfloor n/k\right\rfloor$ is even, and $n \neq k-1$ mod $k$}\\
    d, &\text{if $k$ is even, $\left\lfloor n/k\right\rfloor$ is even, and $n = k-1$ mod $k$}
    \end{cases}
\end{align*}
where $d=s_{2,n}(k(\left\lfloor n/k\right\rfloor-1)-1,2)s_{k,n}(1,\left\lfloor n/k\right\rfloor-2)h_{k,n}(k\left\lfloor n/k\right\rfloor-1)$. Our computer calculations have verified that $a$ and $b$ generate $\Sigma_n$ when $k$ is even and $\A_n$ when $k$ is odd for all pairs $(k,n)$ where $n \geq k \geq 3$, $n \leq 200$ and $k \leq 30$, except for the three cases $(3,6)$, $(3,7)$, and $(3,8)$. These exceptional cases can be handled by a different construction. We have also checked the construction for 1000 additional random pairs of values where $k \leq n \leq 1000$. We have not, however, found a proof that $a$ and $b$ generate $\Sigma_n$ when $k$ is even and $\A_n$ when $k$ is odd.

It would in addition be interesting to determine the likelihood of generating $\Sigma_n$ or $\A_n$ with two (or more) random elements of order $k$ as $n$ goes to infinity, just as Dixon \cite{Dixon} determined for two random elements without order constraints. One could also take up the more restrictive case where the random elements are products of the maximum number of disjoint $k$-cycles. Showing that either $\Sigma_n$ or $\A_n$ is generated with positive probability could also be used as an approach to showing the existence of a generating set of two elements of order $k$.

\begin{problem}
Determine the probability of generating $\Sigma_n$ or $\A_n$ with a fixed number of elements of order $k$.
\end{problem}

The parallels often drawn between $\Mod(S_g)$ and $\Out^+(F_n)$ and $\Out(F_n)$ suggest the following problem.

\begin{problem}
Using Harvey's formula for $\Mod(S_g)$ as a model, produce formulas for $\Out^+(F_n)$ and $\Out(F_n)$ that give the smallest $n$ where an element of order $k$ first appears. Likewise, give formulas for the largest $n$ where these groups fail to contain an element of order $k$.
\end{problem}

Finally, it is known that every finite group embeds in some $\Mod(S_g)$ \cite[Theorem~7.12]{Primer}. One way to state our Lemma \ref{thm:mainlemma} is that for a given $k$, $\Mod(S_g)$ contains as a subgroup the cyclic group of order $k$ for sufficiently large $g$. Since this is true, one obstruction to embedding any particular finite group in $\Mod(S_g)$ for all sufficiently large $g$ is removed.

\begin{question}
For a fixed finite group $G$, does every $\Mod(S_g)$ contain $G$ as a subgroup for sufficiently large $g$? Whenever this is the case, can a formula be given for the largest $g$ where $\Mod(S_g)$ fails to contain $G$ as a subgroup? Whenever $G$ does exist as a subgroup, can it be shown that the elements in a small number of conjugates of $G$ generate all of $\Mod(S_g)$? 
\end{question}

A fundamental result in this direction was shown by Kulkarni \cite{Kulk}: for any finite group $G$, the $g$ for which $G$ acts faithfully on $S_g$ all fall in some infinite arithmetic progression; and further, all but finitely many values in the arithmetic progression are admissible $g$. Additionally, some results in this direction for finite subgroups of $\SO(3,\mathbb{R})$ have been proved by Tucker \cite{Tucker}.

\bibliographystyle{plain}
\bibliography{main}

\end{document}